\documentclass[12pt,reqno,a4paper]{amsart}

\usepackage{mathtools,color, mathabx}
\usepackage{bbold}

\usepackage[colorlinks=true, linkcolor=blue, urlcolor=blue, citecolor=blue, pagebackref=true]{hyperref}

\usepackage[abbrev,bibtex-style]{amsrefs}

\usepackage{amssymb}

\usepackage{amsfonts}

\numberwithin{equation}{section} \theoremstyle{plain}
\newtheorem{theorem}{Theorem}
\newtheorem{proposition}[theorem]{Proposition}
\newtheorem{lemma}[theorem]{Lemma}
\newtheorem{corollary}[theorem]{Corollary}
\newtheorem{definition}[theorem]{Definition}

\theoremstyle{definition}
\newtheorem{remark}[theorem]{Remark}

\renewcommand{\leq}{\leqslant}
\renewcommand{\geq}{\geqslant}

\newcommand\E{\mathbb{E}}

\newcommand\Z{\mathbb{Z}}
\newcommand\R{\mathbb{R}}

\newcommand\C{\mathbb{C}}

\newcommand\F{\mathbb{F}}
\newcommand\Q{\mathbb{Q}}

\newcommand\Alt{\operatorname{Alt}}

\newcommand\supp{\operatorname{supp}}

\newcommand\tr{\operatorname{tr}}

\renewcommand\P{\mathbb{P}}

\newcommand\Res{\operatorname{Res}}
\newcommand\GCD{\operatorname{GCD}}
\newcommand\Gal{\operatorname{Gal}}
\newcommand\Frob{\operatorname{Frob}}
\newcommand\Sym{\operatorname{Sym}}

\renewcommand\Re{\operatorname{Re}}
\renewcommand\Im{\operatorname{Im}}

\renewcommand{\a}{\alpha}
\renewcommand{\b}{\beta}
\renewcommand{\d}{\delta}

\newcommand{\e}{\varepsilon}
\newcommand{\f}{\varphi}
\newcommand{\g}{\gamma}

\newcommand{\s}{\sigma}

\newcommand{\wh}{\widehat}
\newcommand{\wt}{\widetilde}

\newcommand{\cP}{\mathcal{P}}

\newcommand{\cA}{\mathcal{A}}
\newcommand{\cO}{\mathcal{O}}
\newcommand{\cQ}{\mathcal{Q}}
\newcommand{\cE}{\mathcal{E}}

\newcommand{\fp}{\mathfrak{p}}

\DeclareMathOperator{\Err}{Err}

\newcommand\eps{\varepsilon}

\begin{document}

\title[Irreducibility of random polynomials]{Irreducibility of random polynomials of large degree}

\author{Emmanuel Breuillard}
\address{Centre for Mathematical Sciences\\
Wilberforce Road\\
Cambridge CB3 0WA\\
UK }
\email{breuillard@maths.cam.ac.uk}

\author{P\'eter P. Varj\'u}
\address{Centre for Mathematical Sciences\\
Wilberforce Road\\
Cambridge CB3 0WA\\
UK }
\email{pv270@dpmms.cam.ac.uk}

\thanks{EB acknowledges support from ERC Grant no. 617129 `GeTeMo';
PV acknowledges support from the Royal Society.}

\keywords{random polynomials, irreducibility, Riemann hypothesis, Dedekind zeta function, Markov chains}

\subjclass[2010]{11C08 (primary) and 11M41, 60J10 (secondary)}

\begin{abstract}
We consider random polynomials with independent identically distributed
coefficients with a fixed law.
Assuming the Riemann hypothesis for Dedekind zeta functions, we prove that
such polynomials are irreducible and their Galois groups contain the alternating
group with high probability as the degree goes to infinity.
This settles a conjecture of Odlyzko and Poonen conditionally on RH for Dedekind zeta
functions. 
\end{abstract}

\maketitle

\section{Introduction}\label{sc:intro}

Let
\begin{equation}\label{eq:model}
P(x)=x^d+A_{d-1}x^{d-1}+\ldots+A_1x+1\in\Z[x]
\end{equation}
be a random polynomial
with independent coefficients $A_{1},\ldots, A_{d-1}$ taking values
in $0$ and $1$ with equal probability.
Odlyzko and Poonen \cite{OP-01}
conjectured that the probability that $P$ is irreducible in $\Z[x]$
converges to $1$ as $d\to \infty$.

The best known lower bound in this problem is due to Konyagin \cite{Kon-01}
who proved that
\[
\P(\text{$P$ is irreducible})>\frac{c}{\log d}
\]
for an absolute constant $c>0$.

A strongly related problem was studied by Bary-Soroker and Kozma \cite{BSK-210}, who proved that
\[
\P(\text{$x^d+A_{d-1}x^{d-1}+\ldots+A_1x+A_0$ is irreducible})\to 1,
\]
where $A_0,\ldots, A_{d-1}$ are independent random integers uniformly distributed in ${1,\ldots, L}$
for a fixed integer $L$ that has at least $4$ distinct prime divisors.

In another paper, Bary-Soroker and Kozma \cite{BSK-bivariate} studied the problem for bivariate polynomials.
See also \cite{ORW-low-degree} for a study of the probability that a random polynomial has low degree factors,
and \cite{BBBSWW} for computational experiments on related problems.

In this paper we prove the following result.
\begin{theorem}\label{th:main}
Let $P$ be a random polynomial as in \eqref{eq:model}.
Suppose that the Riemann hypothesis holds for the Dedekind zeta function $\zeta_K$
for all number fields of the form $K=\Q(a)$, where $a$ is a root of a polynomial with $0,1$
coefficients.

Then 
\[
\P(\text{$P$ is irreducible in $\Z[x]$})\to 1
\]
as $d\to\infty$.
\end{theorem}

See Section \ref{sc:results} for more precise results, where we discuss the following finer aspects of the problem
\begin{itemize}
\item random polynomials with arbitrary i.i.d.~ coefficients,
\item the rate at which the probability converges to $1$,
\item relaxation of the assumption of RH,
\item Galois groups. 
\end{itemize}

\subsection{Motivation}\label{sc:motivation}

Beyond its intrinsic interest, the problem of irreducibility of random polynomials of high degree is motivated by some other problems, which we now briefly discuss.

It is believed to be computationally difficult to determine the prime factorization
of integers.
On the other hand, polynomial time algorithms are known for computing the
factorization of polynomials in $\Z[x]$.
Given an integer $N\in\Z_{>0}$, we can write it as $N=P(2)$ for a unique
polynomial $P$ with $0,1$-coefficients.
By computing the factorization of $P$ in $\Z[x]$ and evaluating the factors at $2$,
we can obtain a factorization of $N$.

The only weakness of this approach is that the polynomial $P$ may be irreducible
and thus the factorization of $N$ obtained may be trivial.
The problem we study in this paper thus asks for the probability that this procedure returns only a trivial factorization.
Therefore, it is desirable to have results, such as those of this paper, proving that this probability converges to $1$ very fast.

We will discuss our method in Section
\ref{sc:outline}. The method links the problem of irreducibility of random polynomials with mixing times
of certain Markov chains, which are$\mod p$ analogues of the Bernoulli convolutions we had studied in earlier work (see e.g. \cites{breuillard-varju0,varju,breuillard-varju}).
In this paper, we use results available for the Markov chains to study random polynomials,
but this can be reversed. In particular, in a forthcoming paper, we will use the results of this paper to obtain new results about the Markov chains.

Our results on irreducibility assume the Riemann hypothesis for Dedekind zeta functions, or at least some information on the zeros. In our last theorem, Theorem \ref{th:exp-zero}, we show that conversely irreducibility of random polynomials has (modest) implications about the zeros of Dedekind zeta functions.

\subsection{Results}\label{sc:results}

Under the full force of the Riemann hypothesis, our best result is the following.

\begin{theorem}\label{th:GRH}
Let $P=A_dx^d+\ldots+A_1x+A_0\in\Z[x]$ be a random polynomial with independent coefficients.
Assume that $A_1,\ldots, A_{d-1}$ are identically distributed with common law $\mu$.
Assume further that all coefficients are bounded by $\exp(d^{1/10})$ almost surely.
Let $\tau>0$ be a number such that $\|\mu\|_2^2:=\sum_{x \in \Z}\mu(x)^2<1-\tau$.

There are absolute (and effective) constants $c,C>0$ such that if $d \ge C/\tau^4$, then
with probability at least $1-\exp(-c\tau d^{1/2}/\log d)$ the following holds for $P$.
\begin{enumerate}
\item
If RH holds for $\zeta_K$ for $K=\Q$ and for all number fields of the form $K=\Q(a)$, where $a$ is a root of $P$,
then $P=\Phi \wt P$, where  $\wt P$ is irreducible, and
$\Phi$ has $\deg \Phi \leq \frac{C}{\tau} \sqrt{d}$ and is a product of cyclotomic polynomials and $x^m$ for some $m\in\Z_{\ge 0}$.
\end{enumerate}
Moreover with probability at least $1-\exp(-c\tau d^{1/2}/(\log d)^2)$ the following additional property holds for $P$.
\begin{enumerate}
\setcounter{enumi}{1}
\item
If RH holds for $\zeta_K$ for all number fields of the form $K=\Q(a_1,\ldots,a_m)$, where $a_1,\ldots,a_m$
are any number of roots of $P$, then $\Gal(\wt P)\supset \Alt(\deg \wt P)$.
\end{enumerate}
\end{theorem}

Here $\Alt(n)$ denotes the alternating group on $n$ elements, $\Gal(P)$ the Galois group of the splitting field of the polynomial $P$.

There are several remarks in order regarding this theorem.
It is natural to allow that the probability laws of $A_d$ and $A_0$ differ from those of the other coefficients, for example to include
the original problem discussed in the beginning of the paper.
The exponent $\frac{1}{10}$ has no particular significance and the upper bound $\exp(d^{1/10})$ on the coefficients could be relaxed at the expense of some technical complications
in the proof, but we do not pursue this.
Nevertheless, the method of proof definitely requires some upper bound in terms of $d$; it would be interesting to
know if this is also necessary for the theorem to hold.

Our method is based on studying higher order transitivity of the Galois group acting on the roots, and hence it cannot distinguish
between the Galois group being $\Sym(d)$ or $\Alt(d)$.
Deciding whether or not the Galois group is $\Sym(d)$ with probability tending to $1$ appears to be a hard problem.

There are certain obstructions to the irreducibility of $P$ that occur with  probability higher than the estimate
$2\exp(-cd^{1/2}/\log d)$ given in the theorem.
In particular, if $\P(A_0=0)$ is positive, then $x|P$ with positive probability.
Moreover, if $\omega$ is a root of unity, then one may think of $P(\omega)$ as the end point of a random walk on
$\Z[\omega]$ whose steps are given by $A_j\omega^j$ for $j=0,\ldots, d$.
If we fix $\omega\neq 1$ and $\mu$, then for large values of $d$, $\P(P(\omega)=0)$ is proportional to $d^{-r/2}$ (say by the lattice local limit theorem \cite[\S 49, Chap. 9]{GneKol}), where
$r$ is the rank of the lattice $\Z[\omega]$.

In summary, the factor $\Phi$ may be non-trivial with probability higher than $2\exp(-cd^{1/2}/\log d)$,
and its precise behavior can be described by a detailed analysis of random walks on lattices, which we do not pursue here, except for the:

\begin{corollary}\label{cr:00} Let $\mu$ be a probability measure on $\Z$ with finite second moment, which is not supported on a singleton. Let $N$ be a positive integer and $U_N$ be the finite subset of $\C$ consisting of $0$ and all of roots unity $\omega$ with $[\Q(\omega):\Q]<N$.  Let $(A_i)_{i \ge 0}$ be a sequence of i.i.d. random variables with common law $\mu$ and set $P_d=A_dx^d+\ldots+A_1x+A_0\in\Z[x]$. Then, assuming the Riemann hypothesis for Dedekind zeta functions of number fields, as $d\to +\infty$
\[
\P(\text{$P_d$ is irreducible in $\Q[x]$})=1-\P(\exists \omega \in U_N, P_d(\omega)=0) +O_{\mu,N}(d^{-\frac{N}{2}}).
\]
\end{corollary}

In a similar flavor we answer the original problem posed at the beginning of the paper.

\begin{corollary}\label{cr:01}
Let
$P(x)=x^d+A_{d-1}x^{d-1}+\ldots+A_1x+1\in\Z[x]$
be a random polynomial
with independent coefficients $A_{1},\ldots, A_{d-1}$ taking values
in $0$ and $1$ with equal probability.
Suppose that the Riemann hypothesis holds for the Dedekind zeta function $\zeta_K$
for all number fields of the form $K=\Q(a)$, where $a$ is a root of a polynomial with $0,1$
coefficients.

Then 
\[
\P(\text{$P$ is irreducible in $\Z[x]$})=1-\sqrt{\frac{2}{\pi d}}+O(d^{-1}),
\]
where the implied constant is absolute.
\end{corollary}

Polynomials of small Mahler measure can also contribute to the error term. In this respect it is also worth pointing out that an exponential bound in the error term in Theorem \ref{th:GRH}, say of the form $\exp(-c d)$ for some $c>0$ would easily imply the Lehmer conjecture (arguing, say, as in \cite{breuillard-varju0}*{Lemma 16}).

In the proof of part (2) of Theorem \ref{th:GRH}, we will show that the Galois group of $\wt P$ acts $k$-transitively
on its roots with $k > (\log d)^2$. By a well-known fact going back to Bochert and Jordan in the 19-th century, this implies that the Galois group contains the alternating group.
In fact, now there are even better results available, which we will discuss in more details in Section \ref{sc:proofGRH}.
Using the classification of finite simple groups, it has been proved that all $6$-transitive permutation groups contain $\Alt(d)$.
However, if we were to rely on this, it would lead only to a very minor improvement in Theorem \ref{th:GRH}, so
we opted for a proof avoiding the classification.
Unfortunately, our method cannot distinguish between the symmetric and alternating groups.

Bary-Soroker and Kozma proved that if $\mu$ is the uniform distribution on an interval, then with probability tending
to $1$, the Galois group of $P$ contains $\Alt(d)$ provided $P$ is irreducible.
However, our result applies in greater generality and provides a better bound for the probability of exceptions
conditionally on the Riemann hypothesis.

Next, we state two results, where the reliance on the Riemann hypothesis is relaxed at the expense
of a weakening of the bound.

\begin{theorem}\label{th:mild-hyp1}
For any numbers $\tau>0$ and $\a>\b> 3$, there is $c>0$ such that the following holds.
Let $P=A_dx^d+\ldots+A_1x+A_0\in\Z[x]$ be a random polynomial with independent coefficients.
Assume that $A_1,\ldots, A_{d-1}$ are identically distributed with common law $\mu$.
Assume further that all coefficients are bounded by $d^{1/\tau}$ almost surely
and $\|\mu\|_2^2<1-\tau$.

Then with probability at least $1-2\exp(-c(\log d)^{\beta-2})$ the following holds for $P$.
Suppose $\zeta_K$ has no roots $\rho$ with $|1-\rho|<(\log d)^\a/d$ for all $K=\Q(a)$
for any roots $a$ of $P$.
Then $P=\Phi \wt P$, where  $\wt P$ is irreducible, $\deg \Phi \leq c \sqrt{d}$ and
$\Phi$ is a product of cyclotomic polynomials and $x^m$ for some $m\in\Z_{\ge 0}$.
\end{theorem}

We recall the state of the art in our knowledge about the zeros of Dedekind zeta functions
near $1$ to motivate the next  result.
The Dedekind zeta function $\zeta_K$ has at most one zero $\rho$
with $|1-\rho|<4/\log\Delta_K$, where $\Delta_K$ is the discriminant of the number field $K$,
see \cite{Sta-Dedekind}*{Lemma 3}.
If such a zero exists, it must be real,
and we call it the exceptional zero of $\zeta_K$.
The constant $4$ has been improved, see \cite{Kad-Dedekind} for the latest results.
We note that in the setting of Theorems \ref{th:mild-hyp1} and \ref{th:mild-hyp2},
$\log \Delta_K\le Cd\log d$ for a constant $C$ depending only on $\tau$.

The bounds available for the exceptional zero are much weaker.
We know that $\zeta_K$ has no zeros $\rho$ with
\begin{equation}\label{stark-bound}
|1-\rho|<\frac{c}{d\cdot d!|\Delta_K|^{1/d}},
\end{equation}
where $d$ is the degree of $K$ and $c$ is an absolute constant,
see \cite{Sta-Dedekind}*{proof of Theorem 1'}.
However, conditionally on Artin's holomorphy conjecture for Artin L-functions, we know by \cite{Sta-Dedekind}*{Theorem 4}
that $\zeta_K$ has no zeros $\rho$ with
\[
|1-\rho|<\frac{c}{d\log\Delta_K}+\frac{c}{\Delta_K^{1/d}}.
\]

In the next result, we formulate our hypothesis on the zeros of Dedekind functions allowing for
an exceptional zero.

\begin{theorem}\label{th:mild-hyp2}
For any numbers $\tau>0$, $\a> 4$ and $\gamma>1$ such that $\a>2\gamma+2$, there is $c>0$ such that the following holds.
Let $P=A_dx^d+\ldots+A_1x+A_0\in\Z[x]$ be a random polynomial with independent coefficients.
Assume that $A_1,\ldots, A_{d-1}$ are identically distributed with common law $\mu$.
Assume further that all coefficients are bounded by $d^{1/\tau}$ almost surely
and $\|\mu\|_2^2<1-\tau$.

Then with probability at least $1-2\exp(-c(\log d)^{\a-\g-2})$ the following holds for $P$.
Suppose $\zeta_K$ has at most one root $\rho$ with $|1-\rho|<(\log d)^\a/d$ and none with
$|1-\rho|<\exp(-c(\log d)^{\gamma})$ for all $K=\Q(a)$
for any roots $a$ of $P$.
Then $P=\Phi \wt P$, where  $\wt P$ is irreducible, $\deg \Phi \leq c\sqrt{d}$ and
$\Phi$ is a product of cyclotomic polynomials and $x^m$ for some $m\in\Z_{\ge 0}$.
\end{theorem}

Most of the interest in our final result is when we know unconditionally that the random polynomial $P$ is irreducible with
high probability, e.g. in the setting of the work of Bary-Soroker and Kozma \cite{BSK-210} mentioned above.
Then we  obtain as a direct consequence of the following theorem an unconditional improvement on the bound \eqref{stark-bound} for the exceptional zero of the Dedekind zeta function $\zeta_K$ that holds for most number fields $K$, where $K$ is the sampled by setting $K=\Q(a)$ for a root $a$ of the random irreducible polynomial $P$.

\begin{theorem}\label{th:exp-zero}
For every $\tau>0$ and $\alpha>\beta>3$, there is $c>0$ such that the following holds.
Let $P=A_dx^d+\ldots+A_1x+A_0\in\Z[x]$ be a random polynomial with independent coefficients.
Assume that $A_1,\ldots, A_{d-1}$ are identically distributed with common law $\mu$.
Assume further that all coefficients are bounded by $d^{1/\tau}$ almost surely
and $\|\mu\|_2^2<1-\tau$.

Then with probability at least $1-2\exp(-c(\log d)^{\b-2})$ the following holds for $P$.
There is a root $a$ of $P$ that is not a root of unity, such that $\zeta_{\Q(a)}$ has no zeros $\rho$
with $|1-\rho|<\exp(-(\log d)^{\a+1})$.
\end{theorem}

\subsection{An outline of the proof}\label{sc:outline}

Our strategy for proving the results stated above aims at finding information about the distribution
of the degree sequence in the factorization of the random polynomial $P$ in $\F_p[x]$, and then
uses this information to study irreducibility of $P$ in $\Z[x]$ and the Galois group of its splitting field.

Bary-Soroker and Kozma \cite{BSK-210} approximated (in a certain sense)
the degree sequence in the factorization of a polynomial chosen uniformly at random from degree $d$ monic
polynomials in $\F_p[x]$.
It is very plausible that such an approximation holds in greater generality not only for the uniform distribution,
but we do not know how to prove this.
However, we are able to approximate the statistics of the number of low degree factors and this allows us
to gain information about the Galois groups using special cases of the Chebotarev density theorem.

The most relevant density theorem for our purposes is the prime ideal theorem, which has the following
consequence.

\begin{theorem}\label{th:weinberger}
Let $P_0\in\Z[x]$ be a fixed polynomial and let $p$ be a random prime chosen uniformly in a dyadic range $[y,2y)$.
Then
\begin{equation}\label{eq:random-prime}
\E[\text{number of roots of $P_0$ in $\F_p$}]\to\{\text{number of distinct irred. factors of $P_0$}\}
\end{equation}
as $y\to \infty$.
\end{theorem}

This observation was suggested as a basis for an algorithm to compute the number of irreducible factors of
a polynomial by Weinberger \cite{Wei-factors}.

If the Riemann hypothesis holds for $\zeta_K$ for all $K=\Q(a)$, where $a$ is a root of $P_0$, then
the approximation \eqref{eq:random-prime} is valid once $y>C(\e)(\log\Delta_{P_0})^{2+\e}$.
A more precise discussion of these ideas including proofs will be given in Sections \ref{sc:PIT}--\ref{sc:roots-random-fields}.
We note that if we wish to approximate the distribution of the full degree sequence of the factorization
of $P_0$ in $\F_p$ using the Chebotarev density theorem, then we need to take a much larger value for
$y$ even if we assume the Riemann hypothesis for all relevant Dedekind zeta functions.
Indeed, that would require us to replace the discriminant of $P_0$ with the discriminant of its splitting
field in the above bound, which is potentially much larger, and that would not be sufficient for our purposes.

 The next aim of our strategy is to show that
 \begin{equation}\label{eq:random-poly}
 \E[\text{number of roots of $P$ in $\F_{p_0}$}]\approx 1,
 \end{equation}
 where $P$ is a random polynomial in the setting of the above theorems and $p_0$ is a fixed prime in the range
 $[y,2y)$, which is suitably large for the approximation in \eqref{eq:random-prime} to hold.
 
 If we achieve this goal, then we can randomize the polynomial in \eqref{eq:random-prime} and the prime in \eqref{eq:random-poly}
 and compare the right hand sides to obtain
 \[
\E[\text{number of distinct irred. factors of $P$}]\approx 1.
 \]
Since the number of irreducible factors is always a positive integer, Markov's inequality implies that $P$ has
only one irreducible factor with high probability.
When we will give the details of the argument, we will choose a slightly different route by estimating the second
moments and applying Chebyshev's inequality.
Although this is not necessary for Theorems \ref{th:GRH}, \ref{th:mild-hyp1} and \ref{th:mild-hyp2}, it does help in that it is enough to make the assumption on the Dedekind zeta functions only for those polynomials for
which the conclusion holds. On the other hand, the second moment estimates are necessary for Theorem \ref{th:exp-zero}.

To establish \eqref{eq:random-poly}, we fix an element $a\in\F_{p_0}$ and consider an additive random walk
on $\F_{p_0}$ whose $j$-th increment is $A_ja^j$.
The endpoint of this walk is $P(a)$.
If we can show that the walk mixes rapidly, then we can conclude that
\begin{equation}\label{eq:mixing}
\P(P(a)=0)\approx\frac{1}{p_0}.
\end{equation}
Summing up the probabilities for each $a\in\F_{p_0}$ we arrive at \eqref{eq:random-poly}.

The study of random walks of this kind goes back at least to Chung, Diaconis and Graham \cite{CDG-RW},
who considered the case $a=2$.
Their work has been extended in several directions by Hildebrand (see \cite{Hil-thesis}), however he mostly
focused on the case in which $a$ is a fixed integer independent of $p_0$.
In the setting when $a$ may vary with $p_0$, the diameter of the underlying graph was considered by Bukh, Harper
and Helfgott in an unpublished work, see also \cite{Hel-growth}*{footnote 4 on page 372}, that is, they considered how large $d$ needs
to be taken so that the walk reaches every element of $\F_{p_0}$ with positive probability.
Their approach relies on certain estimates of Konyagin \cite{Kon-RW} pertaining to the Waring problem on finite fields
and we will apply the same method.
See also \cite{breuillard-varju-Lehmer}, where the connection between these random walks and Lehmer's conjecture is
explored.

It turns out that the random walk does not mix fast enough for certain choices of the parameter $a$.
Indeed, if $a=0$, then the walk does not mix at all.
Moreover, if $a=1$, then the mixing time (i.e. how large $d$ needs to be taken for \eqref{eq:mixing} to hold)
will be $\approx p_0^2$, as can be seen by the central limit theorem.
A similar issue arises if $a$ has low multiplicative order.
Therefore, it is useful to exclude certain elements of $\F_{p_0}$ from the count.
We say that an element $a\in\F_{p_0}$ is admissible if it is not the root of a cyclotomic polynomial of degree at most $\log p_0$.
We can then modify \eqref{eq:random-prime} by counting admissible primes on the left hand side and non-cyclotomic
factors on the right.
When we give the details of the argument we will exclude from the admissible elements
not only the roots of cyclotomic polynomials but also
the roots of polynomials of very small Mahler measure.
This allows us to obtain improved bounds.

We are able to show that the mixing time is at most $\log p(\log\log p)^{3+\e}$
for most of the parameters $a\in\F_p$ in a sufficiently strong sense required by our
application.\footnote{We can get better results for typical parameters in a weaker sense,
which is not suitable for the purposes of this paper.
These results will appear in a forthcoming paper.}
This allows us to set $y=\exp(d/(\log d)^{3+\e})$ when we apply \eqref{eq:random-prime}.
Even if we disregard the effect of the exceptional zero, our current knowledge about the zeros
of Dedekind zeta functions would require the larger range $y=\exp(Cd\log d)$.
Unfortunately, an argument based on the analysis of the random walks for a fixed parameter $a\in \F_p$
cannot yield a mixing time better than $c\log p$, since the number of points that the random walk
can reach grows exponentially with the number of steps.
To overcome this barrier, one would need to consider the average distribution of the random walk
over the parameters $a\in\F_p$.
This however seems to be exceedingly difficult to study.

There is one last issue that we need to consider.
The above argument cannot distinguish between irreducible polynomials and proper powers.
Indeed, we are able count distinct irreducible factors only.
To show that $P$ is not a proper power with high probability, we show that $P(2)$ is
not a proper power.
To that end, we will use the large sieve together with the classical $a=2$ case of the above
discussed random walks.

Using the above method, we can also obtain information about the Galois group of $P$.
What we discussed so far amounted to showing that the Galois group acts transitively on the
complex roots of $P$.
A more general version of this argument can be used to show that the action is $k$-transitive for large values of $k$, large enough that it forces the Galois group to contain the alternating group.

 Finally, we comment on the proof of Theorem \ref{th:exp-zero}.
 If the Dedekind zeta function has an exceptional zero, then all other zeros are repelled away from
 $1$ by what is known as the Deuring-Heilbronn phenomenon.
 In the context of Theorem \ref{th:weinberger}, this implies that the left hand side of \eqref{eq:random-prime}
is close to zero for a certain range of primes.
This can be contradicted by \eqref{eq:random-poly}.

\subsection{Organization of the paper}\label{sc:organization}

In Sections \ref{sc:PIT}--\ref{sc:roots-random-fields} we discuss the prime ideal theorem and use it to obtain estimates for the average
number of roots of a polynomial in finite fields related to \eqref{eq:random-prime}.
In Sections \ref{sc:RW} and \ref{sc:exp-root-2}, we study equidistribution of random walks,
we revisit Konyagin's estimates in \cite{Kon-RW} and the argument
suggested by Bukh, Harper and Helfgott.
In Section \ref{sc:Lehmer}, we give an upper bound on the probability that the random polynomial $P$ has a factor of small
Mahler measure utilizing some ideas of Konyagin \cite{Kon-01}.
In Section \ref{sc:powers} we use the large sieve to show that $P$ is not the product of a proper power and cyclotomic factors
with high probability.
In Section \ref{sc:proofs} we combine the above ingredients to prove the results stated in Section \ref{sc:results}.

\subsection{Notation}\label{sc:notation}

If $K$ is a number field, we write $d_K$ for its degree and $\Delta_K$ for its discriminant.
If $P\in\Z[x]$ is a polynomial, we write $d_P$ for its degree, $\Delta_P$ for its discriminant
and
\[
M(P)=a_d\prod_{z_j:|z_j|>1}|z_j|
\]
for its Mahler measure, where $a_d$ is the leading coefficient of $P$ and $z_j$ runs through the complex roots
of $P$ taking multiplicities into account.
We recall the estimates
\begin{equation}\label{measure-bounds}
1+c\Big(\frac{\log\log d_P}{\log d_P}\Big)^3\le M(P)\le(a_0^2+\ldots+a_d^2)^{1/2},
\end{equation}
where the upper bound holds for all $P\neq0\in\Z[x]$ and the lower bound holds if in addition $P$ is not
the product of cyclotomic polynomials and $x^m$ for some $m\in\Z_{\ge0}$.
Here $c>0$ is an absolute constant $a_0,\ldots,a_d$ are the coefficients of $P$.
See \cite{Dob-Lehmer} for the inequality on the left hand side and \cite{BG-heights}*{Lemma 1.6.7} for the right hand side.

We write $\sum_p$ for summation over rational primes.

Throughout the paper we use the letters $c$ and $C$ to denote positive numbers whose values may vary at
each occurrence.
These values are effective and numerical: they could, in principle,
be determined by following the arguments.
We will use upper case $C$ when the number is best thought to be large, and lower case $c$ when it is
best thought to be small.
In addition we will use Landau's $O(X)$ notation to denote a quantity that is bounded in absolute value
by a constant multiple of $X$.

\subsection{Acknowledgments}

The authors are grateful to Boris Bukh, Mohammad Bardestani and Peter Sarnak for helpful discussions
on various aspects of this work. 
We thank the referees for their careful reading of our paper and for useful comments and suggestions.

\section{The prime ideal theorem}\label{sc:PIT}

Let $K$ be a number field of degree $d=d_K$ with discriminant $\Delta=\Delta_K$ and denote by $\cO_K$ its ring of integers.
Write $\zeta_K$ for the Dedekind zeta function of $K$.
Write $A(n)=A_{K}(n)$ for the number of prime ideals $\fp\subset\cO_K$ with $N_{K/\Q}(\fp)=n$.

The purpose of this section is to compute the average value of $A(p)$ with respect to suitably
chosen weights supported on primes.
We first consider this question under the assumption that RH holds for $\zeta_K$. In what follows, $\sum_p$ indicates summation over all positive primes in $\Z$.

\begin{proposition}\label{pr:prime-sum-GRH}
Let $X>1$ be a number and let
\[
h_X(u)=
\begin{cases}
2\exp(-X) & \text{if $u\in(X-\log 2,X]$,}\\
0 &\text{otherwise.}
\end{cases}
\]
If RH holds for $\zeta_K$, then
\[
\sum_{p} A(p)\log (p)h_X(\log p)=1+O(X^2\log(\Delta)\exp(-X/2)),
\]
where the implied constant is absolute.
\end{proposition}
\begin{proof}
We write
\[
\psi_K(x)=\sum_{n,m\in\Z_{>0}: n^m\le x} A(n)\log n.
\]
There is an absolute constant $C>0$ such that if $RH$ holds for $\zeta_K$, then for all $x >1$,
\[
|\psi_K(x)-x| \leq C \sqrt{x}( \log x \log \Delta + d (\log x)^2)
\]
See for example \cite{GM-PIT}*{Corollary 1.2}. Applying this for $x=\exp(X)$ and $x=\exp(X)/2$,
we find that
\begin{align*}
\sum_{n=1}^{\infty}\sum_{m=1}^{\infty} A(n)\log (n)h_X(\log n^m)
=&2\exp(-X)(\psi_K(\exp(X))-\psi_K(\exp(X)/2))\\
=&1+O(X^2\log(\Delta)\exp(-X/2)).
\end{align*}
Here we used that $d_K\le C(\log\Delta_K)$ by Minkowski's lower bound on the discriminant.

We estimate the contribution of the summands for which $n^m$ is not a prime.
First we note that for each of these terms, $n^m$ is a proper power, and there
are at most
\[
\exp(X/2)+\exp(X/3)+\ldots +\exp(X/\lceil X\rceil)\le C\exp(X/2)
\]
such numbers between $\exp(X)$ and $\exp(X)/2$.
Each such number can be written in the form $n^m$ in at most $X$ different ways,
and $A(n)\log n\le d_K X$.
Therefore
\[
\sum_{n,m:\text{ $n^m$ is not prime}}A(n)\log (n)h_X(\log n^m)\le CX^2\log(\Delta)\exp(-X/2).
\]
\end{proof}

The purpose of the rest of this section is to formulate a variant of this proposition with a milder
assumption on the zeros of $\zeta_K$.
Readers only interested in the proof of Theorem \ref{th:GRH} may skip to the next section.
Everything that follows is well known and classical.
We begin by recalling the smooth version of the explicit formula.

\begin{theorem}\label{th:explicit}
Let $g\in C^2(\R)$ be a function supported in a compact interval contained in $\R_{>0}$.
Write
\[
\widehat g(s)=\int_{\R} \exp(isu)g(u)du.
\]
Then
\[
\sum_{n=1}^{\infty}\sum_{m=1}^{\infty} A(n)\log(n) g(m\log(n))= \wh g(-i)-\sum_\rho \wh g(-i\rho),
\]
where the summation over $\rho$ is taken over all zeros of $\zeta_K$ (including the trivial ones) taking multiplicities into account.
\end{theorem}

This result is well known but it does not seem to be readily available in this form in standard text books, therefore
we give the very short proof for the reader's convenience.

\begin{proof}
We note that
\[
\frac{\zeta_K'}{\zeta_K}(s)=-\sum_{n=1}^{\infty}\sum_{m=1}^{\infty} A(n)\log(n)n^{-ms}
\]
for $\Re(s)>1$.

Since $g$ is compactly supported and $C^2$, $\wh g$ is holomorphic and $|\wh g(is)|=O(|\Im(s)|^{-2})$
with an implied constant (continuously) depending only on $\Re(s)$.
By the Fourier inversion formula, we have
\begin{align*}
\int_{\Re(s)=2}n^{-s}\wh g(-i s)ds=&\int_{\Re(s)=0}n^{-s}\wh g(-i s)ds\\
=&i\int_{-\infty}^{\infty}\exp(-i t\log n)\wh g(t)dt=2\pi i g(\log n).
\end{align*}
for each $n\in\Z_{>0}$.

Therefore, we have
\[
\frac{-1}{2\pi i}\int_{\Re(s)=2} \frac{\zeta_K'}{\zeta_K}(s) \wh g(-i s) ds=\sum_{n=1}^{\infty}\sum_{m=1}^{\infty} A(n)\log(n)g(m\log n).
\]
Shifting the contour integration to $\Re(s)=-\infty$ we can recover the claimed formula from the residue theorem.
We note that $\supp(g)\subset\R_{>0}$, $\wh g(-is)$ decays exponentially as $\Re(-is)\to \infty$ and leave the verification
of the rest of the details to the interested reader.
\end{proof}

In the next lemma, we introduce the weight functions that we will use and establish some of their properties.
The aim is to find compactly supported weights $g$ such that its Laplace transform $G(s)=\wh g(-is)$ decays fast when $\Re(s)\le 1$
and $s$ is moving away from $1$.
To achieve the optimal decay, it is useful to choose $g$ depending on the distance of $s$ from $1$ where we wish to
make $G(s)$ small.
The construction was inspired by Ingham \cite{Ing-Fourier}.

\begin{lemma}\label{lm:weights}
Let $X\in\R_{>0}$ and let $k\in\Z_{>0}$.
For $r\in\R_{>0}$, write
\[
I_r(u)=
\begin{cases}
\frac{1}{r} & \text{if $u\in[-r/2,r/2]$}\\
0 &\text{otherwise.}
\end{cases}
\]
Let
\begin{align}
g_{X,k}(u)=&\exp(-u)\underbrace{I_{X/2k}*\ldots *I_{X/2k}}_{\text{$k$-fold}}(u-3X/4),\label{eq:gXk}\\
G_{X,k}(s)=&\widehat g_{X,k}(-i s)=\int_{\R}\exp(su) g(u) du.\label{eq:GXk}
\end{align}

Suppose $k\ge 4$ and $X\ge 2k$.
Then $g_{X,k}\in C^2(\R)$ and it is supported in $[X/2,X]$ and we have $g_{X,k}(u)\le \exp(-u)$ for all $u\in\R$.
We have $G(1)=1$ and the following bounds hold for all $s\in\C$ with $\Re(s)\le 1$ and for all
$\s\in(0,1)$ and $X_1>X_2$,
\begin{align*}
0 \le 1-G_{X,k}(\sigma)\le& X(1-\sigma),\\
|G_{X,k}(s)|\le& \Big(\frac{4k}{|1-s|X}\Big)^k,\\
|G_{X,k}(s)|\le&\exp((\Re(s)-1)X/2),\\
\frac{G_{X_1,k}(\s)}{G_{X_2,k}(\s)}\le&\exp(-(1-\s)(X_1-X_2)/4).
\end{align*}
\end{lemma}

\begin{proof}
The claim $\supp g_{X,k}\subset [X/2,X]$ and $g(u)\le \exp(-u)$ follows immediately from its definition
and the assumption $X\ge 2k$.

Note 
\[
\widehat I_{X/2k}(s)=\frac{\exp(isX/4k)-\exp(-isX/4k)}{isX/2k}.
\]
Then $|\wh g_{X,k}(\s)|\le C |\s|^{-k}$ for  $\s\in\R$, where $C$ is a number that depends only on $X$ and $k$,
and it follows that $g_{X,k}\in C^2$ if $k\ge 4$.

We also have
\[
\widehat g_{X,k}(s)=\exp(3i(s+i)X/4)\Big(\frac{\exp(i(s+i)X/4k)-\exp(-i(s+i)X/4k)}{i(s+i)X/2k}\Big)^k.
\]
We can write
\begin{equation}\label{eq:GXk-formula}
G_{X,k}(s)=\exp(3(s-1)X/4)\Big(\frac{\exp((s-1)X/4k)-\exp(-(s-1)X/4k)}{(s-1)X/2k}\Big)^k.
\end{equation}
Taking the limit $s\to 1$, we get $G_{X,k}(1)=1$.
Using the bound
\[
\Big|\frac{\exp(z)-\exp(-z)}{2z}\Big|\le\exp(-\Re(z))
\]
with $z=(s-1)X/4k$, which is valid for $\Re(z)\le 0$, we get
\[
|G_{X,k}(s)|\le\exp((\Re(s)-1)X/2)
\]
if $\Re(s)\le 1$.

Next, we use 
\[
\Big|\frac{\exp(z)-\exp(-z)}{2z}\Big|\le\frac{\exp(-\Re(z))}{|z|}
\]
with $z=(s-1)X/4k$, which is valid for $\Re(z)\le 0$, and we get
\[
|G_{X,k}(s)|\le\frac{\exp((\Re(s)-1)X/2)}{(|s-1| X/4k)^k}
\]
if $\Re(s)\le 1$.
Using $\exp((\Re(s)-1)X/2)\le 1$, we get the claim.

To show
\[
\frac{G_{X_1,k}(\s)}{G_{X_2,k}(\s)}\le\exp(-(1-\s)(X_1-X_2)/4),
\]
it is enough to prove that $F_1'(Y)/F_1(Y)\le -1$ for $Y\ge 0$, where
\[
F_1(Y)=\exp(-3Y)\Big(\frac{\exp(Y/k)-\exp(-Y/k)}{2Y/k}\Big)^k.
\]
(We use the substitution $Y=(1-\s)X/4$ and \eqref{eq:GXk-formula}).

This follows at once, if we show that $F_2'(Z)/F_2(Z)\le 1$ for $Z>0$, where
\[
F_2(Z)=\frac{\exp(Z)-\exp(-Z)}{2Z}.
\] 
To that end, we calculate
\[
\frac{F_2(Z)'}{F_2(Z)}=\frac{\exp(Z)+\exp(-Z)}{\exp(Z)-\exp(-Z)}-\frac{1}{Z}
\]
and observe that $F_2'(Z)/F_2(Z)\le 1$ is equivalent to
\[
2\exp(-Z)\le\frac{\exp(Z)-\exp(-Z)}{Z}.
\]
We note that the left hand side is always less than $2$ and the right hand side is greater than $2$ for $Z>0$.
The latter can be seen, for example by computing the power series expansion of the right hand side.
\end{proof}

We record some well known estimates for the number of roots of $\zeta_K$ near $s=1$.
These go back at least to Stark \cite{Sta-Dedekind}.

\begin{lemma}\label{lm:Stark}
For every  $0\le r\le 1$, we have
\[
|\{\rho:\zeta_K(\rho)=0, |1-\rho|<r\}|\le\frac{3}{2}+3r\log|\Delta_K|.
\]
There is an absolute constant $C>0$, such that for every $r>1$, we have
\[
|\{\rho:\zeta_K(\rho)=0, |1-\rho|<r\}|\le C \log|\Delta_K|+C d_K r\log r.
\]
We count the zeros with multiplicities and include the trivial ones.
\end{lemma}

\begin{proof}
As in the proof of \cite{Sta-Dedekind}*{Lemma 3}, we have
\[
{\sum}'_\rho\frac{1}{\s-\rho}<\frac{1}{\s-1}+\frac{1}{2}\log|\Delta_K|,
\]
where $1<\s\le 2$ is arbitrary and $\sum'_\rho$ indicates summation over an arbitrary subset of non-trivial zeros of $\zeta_K$
(taking multiplicities into account) closed under conjugation.
If $r<1/2$, we take $\s=1+2r$ and consider the zeros $\rho$ that satisfy $|1-\rho|< r$.
For each such $\rho$, we have
\[
\Re\Big(\frac{1}{\s-\rho}\Big)\ge \frac{1}{3r},
\]
which can be seen easily by finding the image of the disk $\{z:|1-z|<r\}$ under the inversion through $1+2r$.
This gives us
\[
\frac{1}{3r}|\{\rho:|1-\rho|<r\}|<\frac{1}{2r}+\frac{1}{2}\log|\Delta_K|,
\]
which yields
\[
|\{\rho:|1-\rho|<r\}|<\frac{3}{2}+\frac{3}{2}r\log|\Delta_K|.
\]

Taking $\s=1+r$ for $1/2\le r\le 1$, the same argument gives
\[
|\{\rho:|1-\rho|<r\}|<2+r\log|\Delta_K|,
\]
which is stronger than our claim since $2r\log|\Delta_K|>1/2$.

We note that the trivial zeros are among the non-positive integers, and each have
multiplicity at most $d_K$.
Moreover, we have
\[
|\{\rho:0<\Re(\rho)<1, |\Im(\rho)-t|\le 1\}|\le C\log(\Delta_K)+Cd_K\log(|t|+2)
\]
for any $t\in\R$ (see e.g. \cite{LO-Chebotarev}*{Lemma 5.4}).
These two facts easily imply the second claim.
\end{proof}

Now we formulate a variant of Proposition \ref{pr:prime-sum-GRH}
under a milder assumption on the zeros.

\begin{proposition}\label{pr:prime-sum}
Let $\a>\b,\tau\in\R_{>0}$.
Let $X=d(\log d)^{-\beta}$ or $X=2d(\log d)^{-\beta}$ and $k=\lfloor (\log d)^{\a-\b}/10\rfloor$.
Let $K$ be a number field of degree at most $d$ and discriminant at most $\exp(\tau^{-1}d\log d)$ in absolute value.
Suppose that $\zeta_K$ has at most one zero $\rho_0$ such that $|1-\rho_0|<d^{-1}(\log d)^{\a}$
Then
\[
\sum_{p} A(p) \log(p)g_{X,k}(\log(p))=1-G_{X,k}(\rho_0)+O(\exp(-c(\log d)^{\a-\b})).
\]
When the exceptional zero $\rho_0$ does not exist the corresponding term should be removed from the formula.
The implied constant and $c$ may depend only on $\a$, $\b$ and $\tau$.
\end{proposition}

\begin{proof}
In what follows, we will assume that $d$ is sufficiently large depending on $\a$, $\b$ and $\tau$.
Otherwise, the claim may be made trivial by a sufficient choice of the constants.

The proof is based on the explicit formula in Theorem \ref{th:explicit}, which gives us
\begin{equation}\label{eq:explicit}
\sum_{n=1}^{\infty}\sum_{m=1}^{\infty} A(n)\log(n) g(m\log(n))= G(1)-\sum_\rho  G(\rho),
\end{equation}
where $g=g_{X,k}$ and $G=G_{X.k}$.

First, we focus on the left hand side of \eqref{eq:explicit} and show that the terms for which $n^m$
is not a prime do not have a significant contribution.
We write
\[
\sum_{n=1}^{\infty}\sum_{m=1}^{\infty} A(n)\log(n) g(m\log(n))
=\sum_{p}^\infty\sum_{l=1}^{\infty} \wt A(p^l) g(\log(p^l)),
\]
where
\[
\wt A(p^l)=\sum_{n,m: n^m=p^l} A(n)\log(n)=\sum_{j:j|l} A(p^j)\log(p^j).
\]

We note that $\wt A(p)=A(p)\log p$ for all primes $p$ and that
\[
\wt A(p^l)\le \sum_{j=1}^\infty j A(p^j)\log(p)\le d_K\log(p).
\]
(The last inequality is an equality if $p$ is unramified in $K$.)
Therefore, we can write
\begin{align}
\Big| \sum_{n=1}^{\infty}\sum_{m=1}^{\infty} A(n)\log(n) &g(m\log(n))
-\sum_{p} A(p) \log(p)g(\log(p))\Big|\nonumber\\
\le &\sum_{p}\sum_{l=2}^{\infty} d_K\log(p) g(\log(p^l))\label{eq:powers}.
\end{align}

Since the support of $g=g_{X,k}$ is contained in $[\exp(X/2),\exp(X)]$
only those terms contribute in \eqref{eq:powers} for which $p^l\in[\exp(X/2),\exp(X)]$.
This means in particular that $\log(p)\le X$ for all such terms and we can write
\[
\eqref{eq:powers}\le d_K X \sum_{p}\sum_{\substack{l: l\ge 2,\\p^l\in[\exp(X/2),\exp(X)]}} p^{-l},
\]
where we also used $g(\log(x))\le x^{-1}$.
Therefore,
\[
\eqref{eq:powers}\le d_K X \Big(\sum_{n=\exp(X/4)}^{\exp(X/2)} n^{-2}+\sum_{n=\exp(X/6)}^{\exp(X/3)} n^{-3}
+\sum_{n=2}^{\exp(X/4)}\;\sum_{l:n^l\in[\exp(X/2),\exp(X)]} n^{-l}\Big).
\]

We note that
\[
\sum_{l:n^l\in[\exp(X/2),\exp(X)]} n^{-l}\le 2\exp(-X/2)
\]
for any $n$, hence
\[
\eqref{eq:powers}\le d_K X (C\exp(-X/4)+C\exp(-X/3)+2\exp(X/4)\exp(-X/2)),
\]
so we can conclude
\begin{align*}
\Big| \sum_{n=1}^{\infty}\sum_{m=1}^{\infty} A(n)\log(n) &g(m\log(n))
-\sum_{p} A(p) \log(p)g(\log(p))\Big|\nonumber\\
\le & C d_K X\exp(-X/4).
\end{align*}

Now we turn to the right hand side of \eqref{eq:explicit} and estimate the contribution
of the zeros $\rho$ that satisfy $|1-\rho|>d^{-1}(\log d)^\a$.
We write
\[
R_j:=\{\rho:\zeta_K(\rho)=0, 2^{j}d^{-1}(\log d)^\a\le |1-\rho|< 2^{j+1}d^{-1}(\log d)^\a\}
\]
for each $j\in\Z_{\ge 0}$.
We think about this as a multiset with each zero contained in it with its multiplicity.
 
By Lemma \ref{lm:Stark}, we have
\[
|R_j|\le C2^j (\log d)^{\a+1}
\]
for each $j$ such that $2^{j+1}d^{-1}(\log d)^\a\le 1$.
Here we use that $\log \Delta_K\le \tau^{-1}d\log d$.
To consider the case $2^{j+1}d^{-1}(\log d)^\a>1$, we note
\[
\log(2^{j+1}d^{-1}(\log d)^\a)\le Cj,
\]
and the second part of the same lemma implies that
\[
|R_j|\le C (j+1) 2^j(\log d)^{\a+1}\le C 2^{2j}(\log d)^{\a+1}.
\]
So this last estimate holds for all $j$.

By Lemma \ref{lm:weights} we know that
\[
|G(\rho)|\le\Big(\frac{4k}{ 2^{j}d^{-1}(\log d)^\a X}\Big)^k\le \exp(-c(j+1)(\log d)^{\a-\b})
\]
for each $\rho\in R_j$.
We combine this with the bounds on $|R_j|$ and obtain the following estimates
provided $d$ is sufficiently large (depending on $\a,\b$ and $\tau$)
\begin{align*}
\sum_{j=0}^\infty\sum_{\rho\in R_j}|G(\rho)|\le &\sum_{j=0}^\infty C2^j (\log d)^{\a+1}\cdot \exp(-c(j+1)(\log d)^{\a-\b})\\
\le & \sum_{j=0}^\infty C\exp(-c(j+1)(\log d)^{\a-\b})\\
\le & C \exp(-c(\log d)^{\a-\b}).
\end{align*}

We recall that $G(1)=1$ and using the above estimate, we write
\[
|G(1)-\sum_\rho  G(\rho)-(1- G(\rho_0))|\le C \exp(-c(\log d)^{\a-\b}),
\]
where $\rho_0$ is the unique zero of $\zeta_K$ with $|1-\rho_0|<d^{-1}(\log d)^\a$
if it exists and the term $G(\rho_0)$ should be omitted from the formula if there is no such zero.
Combining this with the estimate we gave above for the left hand side, we get the claim of the proposition.
\end{proof}

\section{Splitting of prime ideals and roots in finite fields}\label{sc:primes-roots}

In this section, we record some facts about the connection between the number of roots a polynomial has
in finite fields and the way prime ideals split when we extend $\Q$ by adjoining roots of the polynomial.

We fix two numbers $\kappa\in(0,1/100)$ and $X\in\R_{>10}$.

\begin{definition}[admissible polynomial]\label{admissibility}  We say that an irreducible polynomial $R\in\Z[x]$ is $(X,\kappa)$-\emph{admissible} if $M(R)>\exp(\kappa)$ or $\deg R>10X$. Otherwise it is called $(X,\kappa)$-\emph{exceptional}.
\end{definition}

By abuse of language and ease of notation in this section we will simply speak of admissible or exceptional polynomials without reference to $X$ and $\kappa$, which we assume fixed.

Lehmer's conjecture implies that all exceptional irreducible polynomials are either cyclotomic or equal to $x$.
It follows from a result of  Dubickas and Konyagin \cite{DK-Lehmer}*{Theorem 1}
that the number of exceptional polynomials of degree $d$
is at most $\exp(\kappa d)$ if $d$ is larger than an absolute constant independent of $X$.

The reason for excluding polynomials of small Mahler measure is that this will allow us to obtain slightly
better results in Sections \ref{sc:RW} and \ref{sc:exp-root-2}.
We will set the value of $\kappa$ depending on the common law of the coefficients of the random
polynomials so that the probability of a random polynomial having an exceptional and non-cyclotomic
factor will be very small.
This is proved in Section \ref{sc:Lehmer} using the above mentioned estimate for the number of exceptional
polynomials.
We could opt to make only the low degree cyclotomic polynomials exceptional, but this would not
lead to a significant simplification of our arguments.

\begin{definition}[admissible residue]\label{admissible-res} Let $p$ be a prime such that $\log p\in[X/2,X]$. A residue $a\in\F_p$ is said to be $(X,\kappa)$-\emph{admissible} if it is not the root of an $(X,\kappa)$-exceptional irreducible polynomial ${\rm mod}\; p$.
\end{definition}

Again if $X$ and $\kappa$ are fixed, as we assume in this section, we will drop the prefix $(X,\kappa)$ and speak about admissible residues.

Let $P\in\Z[x]$ be a polynomial, $F$ its splitting field and $p$ be a prime. We write $B_P(p)$ for the number of distinct admissible roots of $P$ in $\F_p$. Write $\wt P$ for the product of the admissible irreducible factors of $P$. Note that $\wt P$ is square free. Write $\Omega$ for the set of complex roots of $\wt P$.

Let $m\in \Z_{>0}$ and consider the diagonal action of $G=\Gal(F|\Q)$ on $\Omega^m$. We may decompose $\Omega^m$ into distinct $G$-orbits and for each orbit $O\in \Omega^m/G$ pick one representative $\omega:=(x_1,\ldots,x_m) \in O$ and consider the subfield $K_O=\Q(x_1,\ldots,x_m)$. The isomorphism class of $K_O$ is independent of the choice of the representative $\omega$ in $\Omega$. 

Recall that $A_K(p)$ denotes the number of prime ideals $\mathfrak{p} \subset \mathcal{O}_K$ with norm $p$.

The purpose of this section is to prove the following.

\begin{proposition}\label{pr:ApBp}
Let $P\in\Z[x]$, let $p$ be a prime such that $p\nmid\Delta_{\wt P}$ and 
and $p\nmid\Res(\wt P,R)$ for any exceptional polynomials $R$.
Let $m\in\Z_{>0}$.
Then
\[
B_P(p)^m=\sum_{O \in \Omega^m/G}A_{K_O}(p).
\]
\end{proposition}

Let $F$ be a finite Galois extension of $\Q$ and let $p\in\Z$ be a prime that
is unramified in $F$.
Then we write $\Frob_F(p)$ for the (conjugacy class) of the Frobenius element
in $\Gal(F|\Q)$ at $p$.

We begin by recalling two standard facts.

\begin{lemma}[\cite{Coh-computational-NT}*{Theorem 4.8.13}]\label{lm:Frob}
Let $P\in\Z[x]$ be a polynomial and let $F$ be a finite Galois extension of $\Q$ containing the roots of $P$.
Let $p$ be a prime such that $p\nmid\Delta_P$.
Then there is a bijective correspondence between the cycles of $\Frob_F(p)$ acting on the complex roots of $P$
and the irreducible factors of $P$ in $\F_p$, such that the length of a cycle equals the degree of the corresponding irreducible factor.
\end{lemma}

\begin{lemma}[\cite{Mar-number-fields}*{Chapter 4, Theorem 33}]\label{lm:ApFrob}
Let $F$ be a finite Galois extension of $\Q$ with Galois group $G=\Gal(F|\Q)$ and let $H\leq G$ be a subgroup. 
Let $p\in\Z$ be a prime, which is unramified in the extension $F|\Q$.
Then the number of fixed points of $\Frob_F(p)$ acting on $G/H$
is $A_K(p)$, that is the number of prime ideals in $\cO_K$ of norm $p$, where $K$ is the subfield of $F$ pointwise fixed by $H$.
\end{lemma}

\begin{proof}[Proof of Proposition \ref{pr:ApBp}]
Recall that $F$ is the splitting field of $P$ and  $\Omega\subset \C$ the set of roots of  $\wt P$.
We apply Lemma \ref{lm:Frob} for $\wt P$, and see that the number of fixed points of $\Frob_F(p)$
acting on $\Omega$ is $B_P(p)$.
Here we used that  all roots of $\wt P$
in $\F_p$ are distinct and admissible, because
$p\nmid\Delta_{\wt P}$ and $p\nmid\Res(\wt P,R)$ for any exceptional $R$.
Therefore, $B_P(p)^m$ is the number of fixed points of $\Frob_F(p)$
acting diagonally on $\Omega^m$.

Consider an orbit $O$ of $\Gal(F|\Q)$ in $\Omega^m$ and let $K_O=\Q(x_1,\ldots,x_m)$ for some representative $\omega:=(x_1,\ldots,x_m)$ of $O$. Let $H$ be the stabiliser of $\omega$ in $G$. By the Galois correspondence $K_O$ is the subfield of $F$ fixed by $H$. Hence the number of fixed points of  $\Frob_F(p)$ in $O$ is $A_{K_O}(p)$
by Lemma \ref{lm:ApFrob}. The claim follows.
\end{proof}

\section{Expected number of roots of a polynomial in a random finite field}\label{sc:roots-random-fields}

We combine the results of the previous two sections and deduce the following two results. Below we have kept the notation of Section \ref{sc:primes-roots}. Recall that the function $h_X$ was defined in Proposition \ref{pr:prime-sum-GRH}, that $\Omega$ is the set of roots of $\wt P$, $F$ the splitting field of $P$ and $G=\Gal(F|\Q)$ its Galois group. Given $m\in \Z_{>0}$, $\kappa\in (0,\frac{1}{100})$ and $X>10$ we will denote by $B_P(p)$ is the set of $(\frac{\kappa}{m},mX)$-admissible roots of $P$ in $\F_p$ (see Definition \ref{admissible-res}).

\begin{proposition}\label{pr:Bp-sum-GRH}
Let $d,m\in\Z_{\ge 1}$.
Let $P\in\Z[x]$ be a polynomial with coefficients in $[-\exp(d^{1/10}),\exp(d^{1/10})]$ of degree at most $d$.
Suppose that for every $G$-orbit $O$ on $\Omega^m$, the Dedekind zeta function $\zeta_{K_O}$ of the subfield $K_O\leq F$ satisfies RH. Let $X \ge md^{1/10}$.
Then
\[
\sum_{p} B_P(p)^m\log(p)h_X(\log p )=|\Omega^m/G|+O(\exp(-X/10)).
\]
The implied constant is absolute.
\end{proposition}

\begin{proposition}\label{pr:Bp-sum}
Let $\a>\beta,\tau\in\R_{>0}$.
Let $X=d(\log d)^{-\beta}$ or $X=2d(\log d)^{-\beta}$ and $k=\lfloor(\log d)^{\a-\b}/10\rfloor$.
Let $P\in\Z[x]$ be a polynomial with coefficients in $[-d^{1/\tau},d^{1/\tau}]$ of degree at most $d$.
Suppose that for every $G$-orbit $O \subset \Omega$ the Dedekind zeta function $\zeta_{K_O}$ of the subfield $K_O\leq F$ has at most one
root $\rho_{K_O,0}$ such that $|1-\rho_{K_O,0}|<d^{-1}(\log d)^\a$.
Then
\[
\sum_{p} B_P(p)\log(p)g_{X,k}(\log p)=\sum_{O \in \Omega/G}(1-G_{X,k}(\rho_{K_O,0}))+O(\exp(-c(\log d)^{\a-\b})).
\]
If the exceptional zero $\rho_{K_O,0}$ does not exist for some $O$, then the term $G_{X,k}(\rho_{K_O,0})$ should be omitted from the formula.
The implied constant and $c$ may depend only on $\a$, $\b$ and $\tau$.
\end{proposition}

We will use the next lemma to estimate the number of primes for which the result of the previous section does not hold.

\begin{lemma}\label{lm:discriminant}
Let $P\in\Z[x]$ be a polynomial of degree at most $d$ with coefficients in $[-H,H]$.
Let $Q$ be a polynomial that divides $P$.
Then
\[
|\Delta_{Q}|\le (Hd)^{2d}.
\]
For any irreducible $R\in\Z[x]$ of degree at most $d$ with $M(R)\leq 2$, we have
\[
|\Res(\wt P, R)|\le (4Hd)^{2d}.
\]
Let $K$ be a number field obtained by adjoining at most $m$ roots of $P$ to $\Q$.
Then
\[
|\Delta_K|\le (Hd)^{2md^m}.
\]
\end{lemma}

\begin{proof}Recall Mahler's bound on the discriminant of a polynomial $Q \in \C[x]$ of degree $n$ (\cite{Mah-discriminant}*{Theorem 1})
\begin{equation}\label{mahler-bound} |\Delta_Q| \leq n^nM(Q)^{2n-2}.
\end{equation}
If $Q$ divides $P$, $M(Q)\le M(P)\le H(d+1)^{1/2}$ by \eqref{measure-bounds} hence Mahler's bound gives:
\[
|\Delta_{Q}|\le d^d (H(d+1)^{1/2})^{2d-2}\le (Hd)^{2d}.
\]
Now recall that $|\Delta_{\wt P R}|=|\Delta_{\wt P} \Delta_{R}| \Res(\wt P,R)^2$. Since $R$ is irreducible and $\wt P$ is square free, $|\Delta_{\wt P} \Delta_{R}|\geq 1$ and thus $\Res(\wt P,R)^2\leq |\Delta_{R\wt P}|$. Moreover $M(\wt P R)=M(\wt P)M(R)\le 2M(\wt P)$. So by \eqref{mahler-bound} and \eqref{measure-bounds} we conclude
\[
\Res(\wt P, R)^2\le (2d)^{2d} (2H(d+1)^{1/2})^{4d-2}\le (4Hd)^{4d}.
\]

Let $\a_1,\ldots, \a_m$ be roots of $P$ and $K=\Q(\a_1,\ldots,\a_m)$.
For any two number fields $L_1$, $L_2$ we have
\[
|\Delta_{L_1L_2}|\le |\Delta_{L_1}|^{[L_1L_2:L_1]}|\Delta_{L_2}|^{[L_1L_2:L_2]}\le |\Delta_{L_1}|^{\deg L_2}|\Delta_{L_2}|^{\deg L_1},
\]
see e.g. \cite{Toy-Composite-Discriminant}.
Using this inductively, we can write
\[
|\Delta_K|\le |\Delta_{\Q(\a_1)}|^{d^{m-1}}\cdots |\Delta_{\Q(\a_m)}|^{d^{m-1}}\le|\Delta_P|^{md^{m-1}},
\]
which proves the claim by the first part.
\end{proof}

\begin{proof}[Proof of Proposition \ref{pr:Bp-sum-GRH}]
We apply Proposition \ref{pr:prime-sum-GRH} for each $K_O$.
\begin{align*}
\sum_{O\in\Omega^m/G}\sum_{p} A_{K_O}(p)\log(p)h_X(\log p)=&|\Omega^m/G|+O(d^mX^2\cdot2md^{m+1/10}\exp(-X/2))\\
=&|\Omega^m/G|+O(\exp(-X/10)).
\end{align*}
Here we used the estimate for $\Delta_{K_O}$ from Lemma \ref{lm:discriminant}, and the bound $d^m=O(\exp(X/10))$, which follows from our assumption $X\ge md^{1/10}$.

We proceed to estimate
\begin{equation}\label{eq:AB}
\Big|\sum_{p} B_P(p)^m\log(p)h_{X}(\log(p))-\sum_{O \in \Omega^m/G}\sum_{p} A_{K_O}(p)\log(p)h_{X}(\log(p))\Big|.
\end{equation}

By Proposition \ref{pr:ApBp}, a prime $p$ may contribute to \eqref{eq:AB} only if $p|\Delta_{\wt P}$ or $p|\Res(\wt P,R)$ for some $(\frac{\kappa}{m},mX)$-exceptional irreducible $R$.
As we already noted, \cite{DK-Lehmer}*{Theorem 1} implies that the number of exceptional polynomials is at most
$\exp(10\kappa X)\le \exp(X/10)$.
Therefore, the number of primes $p$ contributing to \eqref{eq:AB} is at most
\[
d^{C}\exp(X/10).
\]
Here we used again the bounds from Lemma \ref{lm:discriminant}.

Since for any $p$ we always have $0\le B_P(p)^m\le d^m$ and
\[
0\le\sum_{O\in \Omega^m/G} A_{K_O}(p)\le \sum_{O\in \Omega^m/G} \deg K_O=\sum_{O\in \Omega^m/G} |O|=|\Omega^m|\le d^m,
\]
the contribution of a prime to \eqref{eq:AB} is at most
$d^m X\cdot 2\exp(-X)$.
Therefore,
\[
\eqref{eq:AB}\le d^{m+C} X\exp(X/10-X)\le C\exp(-X/10),
\]
and the claim follows.
\end{proof}

\begin{proof}[Proof of Proposition \ref{pr:Bp-sum}]
The proof is similar to the previous one.
By Lemma \ref{lm:discriminant}, we have $|\Delta_{K_O}|\le d^{2(1/\tau+1)d}$ for each orbit $O$ in $\Omega$.
Hence Proposition \ref{pr:prime-sum} applies to each $K_O$ and we obtain
\begin{align*}
\sum_{O}\sum_{p} &A_{K_O}(p)\log(p)g_{X,k}(\log p)\\
=&\sum_{O \in \Omega/G}(1-G_{X,k}(\rho_{K,0}))
+O(\exp(-c(\log d)^{\a-\b})).
\end{align*}
Since Mahler measure is multiplicative using Dobrowolski's lower bound \eqref{measure-bounds} the number $|\Omega/G|$ of irreducible factors of $\wt P$ is at most  $|\Omega/G|\le C(\log d)^4$. Hence $|\Omega/G|$ can be absorbed into $\exp(-c(\log d)^{\a-\b})$. We proceed to estimate
\begin{equation}\label{eq:AB2}
\Big|\sum_{p} B_P(p)\log(p)g_{X,k}(\log(p))-\sum_{O \in \Omega/G}\sum_{p} A_{K_O}(p)\log(p)g_{X,k}(\log(p))\Big|.
\end{equation}

We estimate the number of primes $p$ contributing to \eqref{eq:AB2} just as we did in the previous proof and find
that there are at most
\[
2(\tau^{-1}+1)d\log(4d)\exp(X/10)
\]
such primes.

Since $g_{X,k}(p)\le p^{-1}$, each such prime contributes to \eqref{eq:AB2}  at most
$d X\exp(-X/2)$.
Therefore,
\[
\eqref{eq:AB2}\le 2(A+1)d\log(4d)\exp(X/10)\cdot dX\exp(-X/2)\le O(\exp(-c(\log d)^{\a-\b})),
\]
and the claim follows.
\end{proof}

\section{Equidistribution of Random walks}\label{sc:RW}

We study equidistribution of certain random walks in this section.
The basic example of these is the walk on $\F_p$ started at $0$
whose steps are given by $x\mapsto \a x\pm 1$, where $\a\in\F_p$ is a fixed
parameter and the signs $\pm$ are chosen independently at random with
equal probabilities at each step.
The study of related random walks goes back to \cites{CDG-RW, Hil-thesis}, but
those studies are mostly concerned with the case, when $\a$ is a fixed integer
independent of $p$.
Much less is known if $\a$ is allowed to vary with $p$.

We will also need to consider direct products of such walks.
Before introducing our notation for the general case, we first outline the
arguments in the basic setup mentioned above.
We write $\nu_\a^{(d)}$ for the probability measure on $\F_p$ that is the distribution of
the random walk after $d+1$ steps.
It is easily seen that $\nu_\a^{(d)}$ is the law of the random variable $S_d(\alpha):=\sum_{j=0}^{d} X_j\a^j\in\F_p$,
where $X_j\in\{-1,1\}$ are independent unbiased random variables, and we can
write its Fourier transform as
\[
\wh\nu_\a^{(d)}(\xi):=\mathbb{E}\big(\exp(\frac{2i \pi S_d(\alpha)}{p})\big)=\prod_{j=0}^d\cos(2\pi\a^j\xi/p).
\]

Our first aim is to bound $|\wh\nu_\a^{(d)}(\xi)|$ away from $1$.
Expanding $\cos$ in power series at $0$, we see that
we need to give a lower bound for $\sum([\a^j\xi]^\sim)^2$, where $[\cdot]^\sim$ denotes
the unique lift of an element in $\F_p$ to $(-p/2,p/2) \cap \Z$.
We will use some bounds of Konyagin \cite{Kon-RW} for this purpose.
We note that Bukh, Harper and Helfgott have used similar ideas in  unpublished work (see \cite{Hel-growth}*{footnote 4 on page 372}) in order to bound the diameter of the graph underlying the random walk.

Once we have bounded away   $|\wh\nu_\a^{(d)}(\xi)|$ from $1$, we can exploit the fact that $\nu_\a^{(d)}$ is
a convolution product of the form $\nu_\a^{(-1,d_1)}*\cdots*\nu_\a^{(d_k,d)}$, where $0\le d_1\le\ldots d_k\le d$
are some integers and $\nu_\a^{(d_1,d_2)}$ is the law of $\sum_{j=d_1+1}^{d_2} X_j\a^j\in\F_p$.
We can also bound $|\wh\nu_\a^{(d_j,d_{j+1})}(\xi)|$ away from $1$ in a similar manner.
Multiplying these bounds together we can get sufficiently strong bounds for $|\wh\nu_\a^{(d)}(\xi)|$
so that we can deduce that the random walk is equidistributed using Parseval's formula.

Since we do not need an equidistribution result for each individual parameter $\a$, we can improve the above
argument by giving an initial estimate for the sum of the $L^2$ norms:
\begin{align*}
\sum_{\a\in\F_p}\|\nu_\a^{(d)}\|_2^2=\sum_{\a\in\F_p}\frac{1}{2^{2d+2}}|\{(x_0,\ldots, x_{d}),&(x_0',\ldots, x_d')\in\{\pm1\}^{d+1}:\\
&x_0+\ldots+ x_d\a^d=x_0'+\ldots+ x_d'\a^d\}|.
\end{align*}
Such an initial bound can be given by exploiting the fact that the polynomial equation
\[
x_0+\ldots+ x_d\a^d=x_0'+\ldots+ x_d'\a^d
\]
may have at most $d$ roots in $\F_p$ unless $x_j=x_j'$ for all $j$.

The rest of the section is organized as follows.
We set out our general framework in the next section and state the equidistribution result we will use later.
In Section \ref{sc:Konyagin}, we give a generalized exposition of Konyagin's argument in our setup
with some slight quantitative improvements.
Then we use it to deduce an estimate for the Fourier coefficients of $\nu_\a^{(d)}$.
We prove our main equidistribution result (Proposition \ref{pr:equi-dist}) in Section \ref{sc:equi-dist}.
Finally, in Section \ref{sc:a2} we prove another equidistribution statement that we need exclusively to bound the probability that
a random polynomial is a proper power.

In this paper, we focus only on those equidistribution results that we need in our applications.
We believe that these random walks are of independent interest and we will study them further
in a subsequent paper.

\subsection{The general setting and results}\label{sc:gen-nota}

We use the following notation throughout this section.
Let $M\in\Z_{>0}$.
Let $p_1,\ldots, p_M$ be distinct primes (say each $\ge 5$) and let $m_1,\ldots, m_M\in\Z_{>0}$ be numbers. 
Let
\begin{align*}
V=&\bigoplus_{i=1}^M \F_{p_i}^{m_i},\\
D=&\max_{i=1,\ldots,M} m_i,\\
Q=&p_1\cdots p_M,
\end{align*}

For $\a=(\a_1,\ldots,\a_M) \in V$ we write $\a_i:=(\a_{i,1},\ldots,\a_{i,m_i}) \in \F_{p_i}^{m_i}$ and for another $\beta \in V$ we write $\alpha \beta =(\alpha_{i,j} \beta_{i,j})_{i,j}$, so for instance if $n\in\Z$, we have $\a^n=(\a_{i,j}^n)_{i,j}$.

We have a canonical isomorphism of additive groups between  $\oplus_i^M \F_{p_i}$ and $\Z/Q\Z$ given by
\[
\Psi:(x_1,\ldots,x_M) \mapsto \sum_{i=1}^M \psi_i(x_i),
\]
where $\psi_i$ denotes the additive homomorphism $\F_{p_i}  \hookrightarrow \Z/Q\Z$, $x \mapsto \frac{Q}{p_i}x$. Moreover we have the trace map $\tr: V \to \oplus_i^M \F_{p_i}$ given by
\[
\tr(\a)= (\sum_{j=1}^{m_1} \alpha_{1,j},\ldots,\sum_{j=1}^{m_M} \alpha_{M,j}).
\]

Let $X_0,X_1,\ldots, X_d$ be a sequence of independent $\Z$ valued random variables.
We assume that $X_1,\ldots, X_{d-1}$ are identically distributed and write
$\mu$ for their common law. 
We will study the random walk in the additive group $(V,+)$ whose $n$-th step is $\sum_{j=0}^n X_j \a^j$ and
denote by $\nu_\a^{(n)}$ the law of this random element.

Our decision to exempt $X_0$ and $X_d$ from having the same distribution as the other steps of the walk
is motivated by our intention to permit families of random polynomials whose leading and constant
terms have a distribution that differs from the rest.
Our method would allow us to relax the requirement of identical distribution further
by allowing small perturbations of the same law
and a small number of exceptional steps.
We leave it to the interested reader to formulate such a statement.

\begin{definition}
We say that $\a\in V$ is \emph{generic} if  for each $i\in [1,M]$ the coordinates $(\a_{i,j})_{1\le j\le m_i}$ are non-zero and pairwise distinct.
\end{definition}

For $\kappa \in (0,\frac{1}{100})$, we write $\cA_{\kappa}\subset V$ for the set of generic $\a\in V$ such that none of the coordinates $\a_{i,j}$ is a root of a polynomial $P \in \Z[x]$ with $\deg P \le 3\log(Q^D)$ and Mahler measure $M(P)\leq \exp({\kappa})$.

The aim of this section is to prove the following result, which asserts under suitable
conditions that the probability that the random walk is at $0$ after $d$ steps is
approximately $|V|^{-1}$ on average for parameters $\a\in\cA_\kappa$.

\begin{proposition}\label{pr:equi-dist}
There are absolute constants $c,C>0$ such that the following holds.
Let $d\in\Z_{>0}$, $0<\tau<1$.
Suppose that
\begin{align*}
d\ge &C(\kappa \tau)^{-1}MD\log Q^D(\log\log Q^D)^3,\\
\log(Q^D)\ge&\max\Big(\frac{1}{\kappa}, \tau^{-1}\Big),\\
\|\mu\|_2^2:=\sum_{x \in \Z} \mu(x)^2 < &1-\tau.
\end{align*}
Suppose further that $\supp \mu\subset(-p_i/2,p_i/2)$ for each $i=1,\ldots,M$.

Then
\[
\Bigg|\sum_{\a\in A}\nu_\a^{(d)}(0)-\frac{|A|}{|V|}\Bigg|
<\exp\Big(-c\frac{\tau \kappa d}{\log Q^D(\log\log Q^D)^2}\Big)
\]
for any $A\subset\cA_\kappa$.
\end{proposition}

\noindent \emph{Remark.} This proposition will be used in Section \ref{sc:exp-root-2} twice. Once with $M=1$ and a large prime $p$ and fixed power $m$. And another time with $M=2$ and $m_1=m_2$. For the theorems of the introduction, except part (2) of Theorem \ref{th:GRH} about the generic Galois group, it is enough to consider the case $m=1$.

\subsection{Estimates for the Fourier coefficients of the random walk}\label{sc:Konyagin}

The aim of this section is to revisit an argument of Konyagin from \cite{Kon-RW}
to obtain Proposition \ref{pr:Konyagin} below.
Then we will use it in Proposition \ref{pr:Fourier} to deduce a bound for the Fourier coefficients
of the random walk.

For each $\a\in\Z/Q\Z$, we write $\wt \a$ for the unique lift of $\a$ in $\Z\cap(-Q/2,Q/2]$.
\begin{proposition}\label{pr:Konyagin}
Let notation be as in Section \ref{sc:gen-nota}.
Let $\a,\beta\in V$.
Assume that $\a$ is generic and $\beta_{i,j}\neq 0$ for all $i,j$.
Write
\[
S_n=\Psi \circ \tr( \beta \alpha^n)\in \Z/Q\Z.
\]

Let $L\ge 200 \log Q^D\log\log Q^D$ be an integer and suppose that
\begin{equation}\label{eq:hyp-Kon}
\sum_{n=0}^{L} \wt S_n^2\le \frac{Q^2}{8\log(4L)}.
\end{equation}

Then for each $i=1,\ldots, M$ and $j=1,\ldots, m_i$, there is $P_{i,j}\in \Z[x]$ of degree at most $3\log(Q^D)$
with Mahler measure at most $(\log(Q^D))^{30 (\log Q^D)/L}$ such that $P_{i,j}(\a_{i,j})=0$.
\end{proposition}

Write $e_Q(y)=\exp(-2\pi i y/Q)$ for $y \in \Z/Q\Z$.
Given $\beta \in V$ the function $\chi_\beta:V \to \C^\times$
\begin{equation}\label{chibet}
\chi_\beta: x \mapsto e_Q\Big(\Psi \circ \tr (\beta x)\Big)
\end{equation}
is a complex character of the additive group $(V,+)$ and every character is of this form.
Given a measure $\nu$ on $V$, we use the following notation for its Fourier transform:
\[
\wh\nu(\beta)=\sum_{x\in V} \chi_\beta(x) \nu(x).
\]

We write $\nu_\a^{(l_1,l_2)}$ for the law of the random element in $V$ given by
$\sum_{n=l_1+1}^{l_2} X_n\a^n$.
In this notation $\nu_\a^{(d)}=\nu_\a^{(-1,d)}$.

\begin{proposition}\label{pr:Fourier}
Let notation be as in the beginning of the section.
Let $\a,\beta\in V$.
Assume that $\a$ is generic and $\beta_{i,j}\neq 0$ for all $i,j$.
Suppose further that $\supp \mu\subset(-p_i/2,p_i/2)$ for each $i=1,\ldots,M$.

Let $L\ge 200 \log Q^D\log(\log Q^D)$ be an integer and suppose that
there are $i,j$ such that $\a_{i,j}$ is not a root of an integer polynomial of degree at most $3\log Q^D$
with Mahler measure at most $(\log Q^D)^{30 \log (Q^D)/L}$.

Then
\[
|\wh{\nu_\a^{(l_1,l_2)}}(\beta)|\le\exp\Big(-\frac{1-\|\mu\|_2^2}{8\log(4L)}\Big)
\]
for all $0\le l_1<l_2<d$ such that $l_2-l_1>L$.
\end{proposition}

First, we focus on the proof of Proposition \ref{pr:Konyagin}, which closely follows Konyagin \cite{Kon-RW}.
Using a pigeon hole argument, it is easy to find non-zero polynomials $P_1,P_2\in\Z[x]$
of degree at most $\log Q^D$ with $\pm1,0$ coefficients such that $P_1(\a_{i,j})=P_2(\a_{i,j}^q)=0$ for all $i,j$.
Here $q$ is a carefully chosen prime number.
The heart of the argument is the idea that when \eqref{eq:hyp-Kon} holds, it is possible to find $P_1$ and $P_2$ in such a way that for
each $i,j$, there is $P_{i,j}\in\Z[x]\setminus\{0\}$ such that $P_{i,j}(\a_{i,j})=0$ and $P_{i,j}(x)|\GCD(P_1(x),P_2(x^q))$.
From this, we will conclude that $\deg(P_{i,j})\le\deg(P_1)$ and $M(P_{i,j})\le M(P_2)^{1/q}$.

We begin to implement this strategy.
Given  a monic irreducible polynomial $P\in\Z[x]$, the next lemma allows us to find a prime $q$ of controlled size
such that whenever $P(x)|Q(x^q)$ for another polynomial $Q\in\Z[x]$, we have $M(P)\le M(Q)^{1/q}$. 

\begin{lemma}\label{lm:q}(See also \cite{Kon-RW}*{Lemma 2})
Let $\a_1,\ldots, \a_n$ be the roots of an irreducible polynomial $P\in\Z[x]$.
Let $s\ge 4\log n$ be a number.
If $n$ is larger than some absolute constant, then there is a prime $q \in (s,2s]$ such that $\a_i/\a_j$ is not a $q$-th root of unity for any
$i\neq j$.
\end{lemma}

\begin{proof}
Denote by $\mathcal{P}_s$ the set of primes between $s$ and $2s$.
Write $\mathcal{R}$ for the collection of integers $r$ such that there is a root of unity of order $r$
among the numbers $\a_i/\a_j$. Suppose by way of contradiction  that $\mathcal{P}_s \subset \mathcal{R}$.

We begin with the observation that if $r\in \mathcal{R}$, then there is $1\le j \le n$ such that $\a_1/\a_j$ is a root of unity of order $r$.
Indeed, let $i,j$ be such that $\a_i/\a_j$ is a root of unity of order $r$ and let
$\s$ be an automorphism of $\overline \Q$ such that $\s(\a_i)=\a_1$.
Then $\s(\a_i/\a_j)=\a_1/\s(\a_j)$ is a root of unity of order $r$ and this proves the claim.

Suppose $r_1,r_2\in \mathcal{R}$ are coprime integers.
We prove $r_1r_2\in \mathcal{R}$.
Let $i$ and $j$ be such that $\a_1/\a_i$ and $\a_1/\a_j$ are roots of unity of order $r_1$ and $r_2$ respectively.
Then $\a_i/\a_j=(\a_1/\a_i)^{-1}(\a_1/\a_j)$ is a root of unity of order $r_1r_2$, which proves the claim.

Therefore, each divisor $r$ of $\prod_{p \in \mathcal{P}_s} p$ belongs to
$\mathcal{R}$. Since the set of roots $\alpha_1,\ldots,\alpha_n$ is invariant under the action of the Galois group, it follows that for any such $r$ all roots of unity of order $r$ appear among the $\alpha_i/\alpha_j$'s.
Hence
\[
n^2\ge \prod_{p \in \cP_s} p =  \exp(\sum_{p \in \cP_s}\log p).
\]
By the prime number theorem, we have
\[
n^2\ge \exp(s/2)
\]
if $n$ and hence $s$ is sufficiently large.
(In fact, we could put here any constant less than $1$ in place of $1/2$.)
This proves the lemma.
\end{proof}

Let $N \ge e \ge 0$ be integers and $X=(x_0,\ldots, x_N)$ a sequence of integers. Following Konyagin \cite{Kon-RW}, we write $\Lambda_e(X)$ for the set of polynomials $P(x)=a_0+ a_1 x + \ldots +a_ex^e\in\Z[x]$
of degree at most $e\in \Z_{\ge 0}$ such that
\[
a_0 x_{j}+\ldots+ a_e x_{j+e}=0
\]
holds for all $j=0,\ldots,N-e$.
We denote by $\Lambda(X)$ the set of polynomials $P$ of degree at most $N$ such that $P\in\Lambda_{\deg P}(X)$.
We note that $P\in\Lambda_e(x_0,\ldots, x_N)$ if and only if $P\in\Lambda(x_0,\ldots,x_{N-e+\deg P})$.

If $X$ were an infinite sequence, then $\Lambda$ would give rise to a principal  ideal in $\Z[x]$.
We need a weaker form of this fact that is valid for finite sequences. 
To this end, we recall the following result.

\begin{lemma}[Konyagin \cite{Kon-RW}*{Lemma 5}]\label{lm:Lambda}
Let $X=(x_0,\ldots, x_N)$ be a sequence of integers and let $P_1,P_2\in\Lambda(X)$.
If $\deg P_1+\deg P_2\le N$, then we have $\gcd(P_1,P_2)\in\Lambda(X)$.
\end{lemma}

\begin{corollary}\label{cr:Lambda}
Let $X=(x_0,\ldots, x_N)$ be a sequence of integers and suppose that $\Lambda(X)$
contains a non-zero polynomial of degree at most $N/2$.
Then there is a unique (up to multiplication by $\pm1$) non-zero polynomial $P_0\in\Lambda(X)$ of minimal degree and with relatively prime coefficients.
Furthermore, a polynomial $P\in\Z[x]$ of degree at most $N-\deg P_0$ is contained in $\Lambda(X)$ if and only if $P_0|P$.
\end{corollary}

In the proof of Proposition \ref{pr:Konyagin} below, we will use the above results for the sequence $x_n=\wt S_n$.
Under the hypothesis \eqref{eq:hyp-Kon}, we will show that there are many polynomials $P(x)=a_0+ \ldots +a_e x^e$
such that $|a_0 \wt S_{j}+\ldots+a_e \wt S_{e+j}|<Q$ for all $j$ in a certain range.
Using the pigeonhole principle, we will find a polynomial that in addition satisfies
\begin{equation}\label{eq:modQ}
a_0 S_{j}+\ldots+a_e  S_{e+j}=0
\end{equation}
in the same range for $j$'s.
These two properties imply that $P\in\Lambda(X)$ for $X=(S_n)$. The next lemma shows that it is enough to satisfy \eqref{eq:modQ} for a smaller range of $j$'s.

\begin{lemma}\label{lm:modp}
Let $m\in\Z_{>0}$ and let $p$ be a prime.
For each $1\le j\le m$, let $\a_j,\b_j\in\F_{p}$.
Write
\[
T_n=\sum_{j=1}^m \b_j\a_j^n
\]
for each $i\in\Z_{\ge 0}$.
Let $P(x)=a_0+ \ldots +a_ex^e\in\F_p[x]$ be a polynomial.

Suppose that the elements $\a_j$ are pairwise distinct
and $\beta_j\neq 0$ for all $j$.
Suppose further that
\[
a_0T_{n}+a_1T_{n+1}+\ldots+a_{e}T_{n+e}=0
\]
for all $n=0,\ldots,m-1$.

Then $P(\a_j)=0$ for all $1\le j\le m$
and
\[
a_0T_{n}+a_1T_{n+1}+\ldots+a_{e}T_{n+e}=0
\]
for all $n\in\Z_{\ge 0}$.
\end{lemma}

\begin{proof}
The hypothesis of the lemma implies that
\begin{equation}\label{eq:hyp}
\sum_{j=1}^m\b_j \a_j^nP(\a_j)=0,
\end{equation}
for each $n=0,\ldots, m-1$.

We note that the vectors $(\b_1 \a_1^n,\ldots, \b_m \a_m^n)$ for $n=0,\ldots, m-1$ are
linearly independent, as can be seen using Vandermonde determinants.
Hence the system of linear equations \eqref{eq:hyp} in $P(\a_j)$
as variables has only the trivial solution, that is $P(\a_j)=0$ for each $j$, which proves the first claim.
In addition,
\[
a_0T_{n}+a_1T_{n+1}+\ldots+a_{e}T_{n+e}=\sum_{j=1}^m \b_j\a_j^n P(\a_j)=0
\]
for each $n\in\Z_{\ge 0}$ and this establishes the second claim.
\end{proof}

\begin{proof}[Proof of Proposition \ref{pr:Konyagin}]
Set $E=3\lfloor \log Q^D\rfloor$. Note that  $1\leq E\leq L/3$ and $2^E > Q^D \ge 3$ and  $\lfloor L/6E\rfloor \ge 4\log E$.
Our first aim is to show that there is a polynomial $P_1\neq 0$ of degree at most $E$ 
such that $P_1\in\Lambda_{E}(\{\wt S_n\}_{n=0}^L)$.

Let $\xi_0,\ldots, \xi_E$ be a sequence of independent, unbiased $\pm1$ valued
random variables.
For any $n=0,\ldots, L-E$, we have
\[
\P\Big(\Big|\sum_{j=0}^{E} \xi_j \wt S_{j+n}\Big|\ge \frac{Q}{2}\Big)\le 2\exp\Big(-\frac{(Q/2)^2}{2 \sum _0^E  \wt S_{j+n}^2 }\Big)\le \frac{1}{2L}
\]
by Hoeffding's inequality and our assumption \eqref{eq:hyp-Kon}.

Therefore
the set
\begin{align*}
\Omega:=\Big\{x=(x_0,\ldots,x_E)\in\{-1,1\}^{E+1}:&\Big|\sum_{j=0}^{E} x_j \wt S_{j+n}\Big|< \frac{Q}{2}\\
&\text{for all $n=0,\ldots, L-E$}\Big\}
\end{align*}
has cardinality more than $2^{E+1}/2> Q^D$.

By the pigeonhole principle, it follows that there are $x\neq y\in\Omega$
such that
\[
\sum_{j=0}^{E} x_j \wt S_{j+n}=\sum_{j=0}^{E} y_j \wt S_{j+n}
\]
for all $n=0,\ldots,D-1$.
We set $a_j=(x_j-y_j)/2\in \{-1,0,1\}$. It follows from Lemma \ref{lm:modp} applied $M$ times to each $T^i_n:=\sum_{j=1}^{m_i} \beta_{i,j}\alpha_{i,j}^n$ with the polynomial $P=a_0+\ldots+a_Ex^E \in \Z[x]$ that
\[
\sum_{j=0}^{E} a_j  S_{j+n}=0 \in \Z/Q\Z
\]
for all $n\in\Z_{\ge 0}$.
Since $x,y\in\Omega$ and $a_j=(x_j-y_j)/2$, we know that
\[
\Big|\sum_{j=0}^{E} a_j \wt S_{j+n}\Big|<Q/2
\]
and hence
\[
\sum_{j=0}^{E} a_j \wt S_{j+n}=0
\]
for any $n=0,\ldots, L-E$.
This means that
\[
P_1(x):=a_E x^E+\ldots+a_1x+a_0\in \Lambda_{E}(\{\wt S_n\}_{n=0}^L)\subset \Lambda(\{\wt S_n\}_{n=0}^{\lceil 2L/3\rceil}),
\]
because $E\le L/3$.

Since $\lceil 2L/3\rceil\ge 2E$, by Corollary \ref{cr:Lambda}
there is a unique (up to multiplication by $\pm1$) polynomial $P_0\in\Lambda(\{\wt S_n\}_{n=0}^{\lceil 2L/3\rceil})$
with relatively prime coefficients and of minimal degree.
By Lemma \ref{lm:modp}, $P_0(\a_{i,j})=0$ for all $i$ and $j$.
Then for each $i$ and $j$ there is an irreducible factor  $P_{i,j}$ of $P_0$ over $\Z$
such that $P_{i,j}(\a_{i,j})=0$.

Fix $i$ and $j$. We already know that $\deg P_{i,j}\le E$, and we set out to prove $M(P_{i,j})\le (D\log Q)^{30 D\log Q/L}$.
Write $\{\b_k\}_{k=1}^{\deg P_{i_j}}$ for the roots of $P_{i,j}$.
We set $s=\lfloor L/6E\rfloor$.
Since $L\ge 200 D\log Q\log(D\log Q)$, we have $s\ge 4\log E$.
By Lemma \ref{lm:q}, there is a prime $q\in (s,2s]$ such that $\b_k/\b_l$ is not a $q$-th root of unity for any $k\neq l$.
This means that the numbers $\{\b_k^q\}_{k=1}^{\deg P_{i,j}}$ are all distinct.
Note that $q\le 2s\le L/3E$.

We now employ the same argument as above and find a non-zero polynomial $P_2$ of the form $P_2(x)=Q_2(x^q)$ for some $Q(x)=b_0+\ldots+b_Ex^E \in \Z[x]$ with $b_j\in \{-1,0,1\}$ for $j=1,\ldots,E$ and 
\[
\sum_{k=0}^{E} b_k \wt S_{kq+n}=0
\]
for any $n=0,\ldots, L-Eq$. Hence $P_2\in\Lambda_{Eq}(\{\wt S_n\}_{n=0}^L)\subset \Lambda(\{\wt S_n\}_{n=0}^{\lceil 2L/3\rceil})$, because $Eq\le L/3$.

Since $\lceil 2L/3\rceil\ge 2 Eq$, we have $P_0|P_2$ by Corollary \ref{cr:Lambda}.
This means that $\{\beta_k^q\}_{k=1}^{\deg P_{i,j}}$ are all roots of the polynomial $Q_2$, and since they are all distinct we get:
\[
M(P_{i,j})^q\le M(Q_2)\le (E+1)^{1/2}.
\]
where the right hand side follows from \eqref{measure-bounds}. Since $q>s=\lfloor L/6E\rfloor$, $2q\ge L/3E$, we get
\[
M(P_{i,j})\le (E+1)^{1/2q}\le(E+1)^{3E/L}\le (D\log Q)^{30 D\log Q/L}.
\]
\end{proof}

\begin{remark}\label{rm:Komlos}
Let $u_1,\ldots,u_m\in \R^n$ be a sequence of vectors with $\|u_j\|\le 1$.
A conjecture of Koml\'os asserts that there is an absolute constant $C$ such that
for each such sequence of vectors, there is a sequence of signs $\omega_j=\pm1$
such that $\|\omega_1u_1+\ldots+\omega_mu_m\|_\infty<C$.
In this remark, we point out that if this conjecture holds, then assumption \eqref{eq:hyp-Kon}
in Proposition \ref{pr:Konyagin}, may be relaxed to an upper bound of the form $cQ^2$, where
$c$ is an absolute constant.
Unfortunately, the best known result towards Koml\'os's conjecture in \cite{Ban-Komlos}
yields no improvement.

We take a sequence $y=(y_0,\ldots, y_{E})\in\{0,1\}^{E+1}$
and we will apply the conjecture to the vectors $u_j=(y_j\wt S_j,\ldots, y_{j}\wt S_{L-E+j})$ for $j=0,\ldots, E$.
Under the weakened hypothesis $\sum \wt S_n^2\le cQ^2$, Koml\'os's conjecture implies that for each choice of $y$,
there is $\omega(y)=(\omega_0(y),\ldots, \omega_{E}(y))\in\{\pm1\}^{E+1}$ such that
\[
\Big|\sum_{j=0}^E \omega_j(y)y_j\wt S_{j+i}\Big|<Q/2
\] 
for all $i=0,\ldots, L-E$.
Now we see that the collection of sequences of the form $(\omega_0(y)y_0,\ldots, \omega_{E}(y)y_E)\in\{-1,0,1\}^E$
may be used in the place of $\Omega$ in the proof of Proposition \ref{pr:Konyagin} to obtain the same result
under the weaker hypothesis.
\end{remark}

Now we turn to the proof of Proposition \ref{pr:Fourier}.

Given a measure $\mu$ on $\Z$ and $x\in V$ we write $\mu.\d_x$ for the measure on $V$ defined by
\[
\sum_{a\in\Z} \mu(a)\d_{ax}.
\]
With this notation, we can write
\[
\nu_\a^{(l_1,l_2)}=\mu.\d_{\a^{l_1+1}}*\ldots*\mu.\d_{\a^{l_2}}
\]
for any $0\le l_1<l_2<d$.

Recall the notation $\wt a$, which is the unique integer representative of $a\in\Z/Q\Z$ in the interval $(-Q/2,Q/2]$.
For typographical convenience we will also use the notation $[a]^{\sim}$ with the same meaning.

\begin{lemma}\label{lm:F-single}
Let $\mu$ be a measure on $\Z$ and let $x,\beta\in V$.
Then
\[
|\wh{\mu.\d_x}(\beta)|\le
\exp\Big(-\sum_{a_1,a_2\in\Z} \mu(a_1)\mu(a_2)
\Big(\Big[\Psi\circ \tr ((a_1-a_2)\b x)\Big]^{\sim}\Big)^2/Q^2 \Big).
\]
\end{lemma}
\begin{proof}
By definition, we have
\[
|\wh{\mu.\d_x}(\beta)|^2=\sum_{a_1,a_2\in\Z}\mu(a_1)\mu(a_2) \chi_\beta((a_1-a_2)x)
\]
Since $|\wh{\mu.\d_x}(\beta)|^2\in\R$ and 
$\Re(e_Q(a))\le 1-2\wt{a}^2/Q^2$ if $a \in \Z/Q\Z$, via \eqref{chibet} we get:
\[
|\wh{\mu.\d_x}(\beta)|^2 \leq 1 - 2\sum_{a_1,a_2\in\Z}\mu(a_1)\mu(a_2) \Big[\Psi\circ \tr ((a_1-a_2)\b x)\Big]^{\sim}\Big)^2/Q^2 
\]
On the other hand $1-t \leq \exp(-t)$ if $t \in [0,1]$, so the claim follows.
\end{proof}

\begin{proof}[Proof of Proposition \ref{pr:Fourier}]
For $\g\in V$, write as earlier
\[
S_n(\g)=\Psi \circ \tr(\g \a^n) \in \Z/Q\Z.
\]
Using
\[
|\wh{\nu_\a^{(l_1,l_2)}}(\beta)|=\Big|\prod_{n=l_1+1}^{l_2}\wh{\mu.\d_{\a^n}}(\b)\Big|,
\]
Lemma \ref{lm:F-single} implies 
\[
|\wh{\nu_\a^{(l_1,l_2)}}(\beta)|\le\exp\Big(-\sum_{n=l_1+1}^{l_2}\sum_{a_1,a_2\in\Z} \mu(a_1)\mu(a_2) (\wt S_n((a_1-a_2)\b))^2/Q^2\Big).
\]

Using $l_2-l_1\ge L$ and then Proposition \ref{pr:Konyagin}, we can write
\[
\sum_{n=l_1+1}^{l_2}(\wt S_n((a_1-a_2)\beta))^2\ge
\sum_{n=0}^{L}(\wt S_n((a_1-a_2)\beta\a^{l_1+1}))^2>\frac{Q^2}{2\log(4L)}
\]
for all pairs $a_1\neq a_2$.
Since $\supp\mu\subset(-p_i/2,p_i/2)$ for all $i$, we know that $a_1-a_2$ is non zero in $\F_{p_i}$ for all $i$ whenever
$a_1\neq a_2$ in $\Z$.

We note
\[
\sum_{a_1\neq a_2} \mu(a_1)\mu(a_2)=1-\|\mu\|_2^2.
\]
Therefore
\[
|\wh{\nu_\a^{(l_1,l_2)}}(\beta)|\le\exp\Big(-(1-\|\mu\|_2^2)\frac{Q^2}{2\log(4L)}\cdot\frac{1}{Q^2}\Big),
\]
as claimed.
\end{proof}

\subsection{Proof of Proposition \ref{pr:equi-dist}}\label{sc:equi-dist}

We note that $\nu(0)=\sum_{\beta \in V} \wh\nu(\beta)/|V|$ for any measure $\nu$ on $V$. Hence, for any $A\subset\cA_\kappa$:
\begin{equation}\label{eq:target}
\Bigg|\sum_{\a\in A}\nu_\a^{(d)}(0)-\frac{|A|}{|V|}\Bigg|
\le \frac{1}{|V|}\sum_{\a\in A}\sum_{\beta\in V\backslash\{0\}}|\wh\nu_\a^{(d)}(\beta)|.
\end{equation}

We begin by finding a preliminary estimate for
\[
\sum_{\a\in A}\|\nu_\a^{(d_1,d_2)}\|^2_2\le
\sum_{\a\in V}\|\nu_\a^{(d_1,d_2)}\|^2_2.
\]
and then use the Cauchy-Schwarz inequality to convert it into an estimate on the
right hand side of \eqref{eq:target}.
Let $(A_n)_{n=0,\ldots, d}$ and
$(A_n')_{n=0,\ldots,d}$ be  sequences of independent random variables with the same law as $(X_n)_{n=0,\ldots,d}$.
We observe that
\begin{align*}
\sum_{\a\in V}\|\nu_\a^{(d_1,d_2)}\|^2_2=&\sum_{\a\in V}\P(A_{d_1+1}\a^{d_1+1}+\ldots+A_{d_2}\a^{d_2}
=A_{d_1+1}'\a^{d_1+1}+\ldots+A_{d_2}'\a^{d_2})\\
=&\E(\#\{\a\in V: A_{d_1+1}\a^{d_1+1}+\ldots+A_{d_2}\a^{d_2}=A_{d_1+1}'\a^{d_1+1}+\ldots+A_{d_2}'\a^{d_2} \}).
\end{align*}
If $A_j\neq A_j'$ for at least one $j\in(d_1,d_2]$, then the polynomial
\[
(A_{d_1+1}-A_{d_1+1}')x^{d_1+1}+\ldots+(A_{d_2}-A_{d_2}')x^{d_2}
\]
has at most $d_2-d_1$ roots in any given field. This means that for such $A_j$ and $A_j'$
\[
\#\{\a\in V: A_{d_1+1}\a^{d_1+1}+\ldots+A_{d_2}\a^{d_2}=A_{d_1+1}'\a^{d_1+1}+\ldots+A_{d_2}'\a^{d_2} \}
\le (d_2-d_1)^{MD}
\]
and
\[
\sum_{\a\in V}\|\nu_\a^{(d_1,d_2)}\|^2_2\le (d_2-d_1)^{MD} +  |V| \sum_{a \in \Z} \mu(a)^{2(d_2-d_1)} \leq (d_2-d_1)^{MD} +  |V|(1-\tau)^{(d_2-d_1)}.
\]

We set $d_0=\lceil -\log(|V|)/\log(1-\tau)\rceil$, and obtain
\[
\frac{1}{|V|}\sum_{\a\in V}\sum_{\beta\in V}|\wh\nu_\a^{(d,d+d_0)}(\beta)|^2=\sum_{\a\in V}\|\nu_\a^{(d,d+d_0)}\|^2_2\le2 d_0^{MD}
\]
for all $d$.
We note that $\nu^{(2d_0)}_\a=\nu^{(-1,0)}_\a*\nu^{(0,d_0)}_\a*\nu^{(d_0,2d_0)}_\a$.
Therefore for each $\a\in V\setminus\{0\}$, we have, since $|\wh\nu_\a^{(-1,0)}(\beta)|\leq 1$,
\begin{align*}
\frac{1}{|V|}\sum_{\beta\in V}|\wh\nu_\a^{(2d_0)}(\beta)|
\le&\frac{1}{|V|}\sum_{\beta\in V}|\wh\nu_\a^{(0,d_0)}(\beta)\cdot\wh\nu_\a^{(d_0,2d_0)}(\beta)|\\
\le&\Big[\frac{1}{|V|}\sum_{\beta\in V}|\wh\nu_\a^{(0,d_0)}(\beta)|^2\Big]^{1/2}\cdot
\Big[\frac{1}{|V|}\sum_{\beta\in V}|\wh\nu_\a^{(d_0,2d_0)}(\beta)|^2\Big]^{1/2}
\end{align*}
This gives us by another application of Cauchy-Schwarz
\begin{align}\label{eq:F-L1}
\frac{1}{|V|}\sum_{\a\in A}&\sum_{\beta\in V}|\wh\nu_\a^{(2d_0)}(\beta)|\nonumber\\
\le&\Big[\frac{1}{|V|}\sum_{\a\in A}\sum_{\beta\in V}|\wh\nu_\a^{(0,d_0)}(\beta)|^2\Big]^{1/2}\cdot
\Big[\frac{1}{|V|}\sum_{\a\in A}\sum_{\beta\in V}|\wh\nu_\a^{(d_0,2d_0)}(\beta)|^2\Big]^{1/2}\nonumber\\
\le& 2d_0^{MD}.
\end{align}

Now we set $d_1=\lceil \frac{200}{\kappa}\log (Q^D)\log\log (Q^D)\rceil$. If $\alpha \in A \subset\cA_{\kappa}$, then
$\a_{i,j}$ is not a root of a polynomial of degree at most $3\log Q^D$ with Mahler measure
at most $(\log Q^D)^{30\log (Q^D)/d_1}$, and we also have
$d_1\ge 200 \log Q^D\log(\log Q^D)$. Therefore, we can apply Proposition \ref{pr:Fourier} with $L=d_1$ and get
\begin{equation}\label{eq:F-Linfty}
|\wh \nu_\a^{(d,d+d_1)}(\beta)|\le \exp\Big(-\frac{\tau}{8\log(4d_1)}\Big)
\end{equation}
for all $d$, $\a\in \cA_\kappa$ and $\beta\in V\backslash\{0\}$.
(If $\beta$ has some $0$ coordinates, then $V$ splits as a direct sum $V=V_0 \oplus V_1$, with $\beta \in V_1$ having no non-zero coordinate in $V_1$, and  we need to apply the proposition to $V_1$ and the projected random walk on $V_1$ modulo $V_0$. )

Now suppose that $d> 2d_0+K d_1$ for some $K \in \Z_{\ge 0}$ and write
\[
|\wh\nu^{(d)}_\a(\beta)|\le|\wh\nu_\a^{(2d_0)}(\beta)|\cdot|\wh\nu_\a^{(2d_0,2d_0+d_1)}(\beta)|\cdots|\wh\nu_\a^{(2d_0+(K-1)d_1,2d_0+Kd_1)}(\beta)|.
\]
We combine \eqref{eq:F-L1} with \eqref{eq:F-Linfty} and obtain
\[
\frac{1}{|V|}\sum_{\a\in \cA_{\kappa}}\sum_{\beta\in V}|\wh\nu_\a^{(d)}(\beta)|\le2 d_0^{MD}\exp\Big(-K\frac{\tau}{8\log(4d_1)}\Big).
\]

By the assumption on $d$ in the proposition, we can take $K>d/2d_1$ and a simple calculation yields that
\[
2d_0^{MD}<\exp\Big(K\frac{\tau}{16\log(4d_1)}\Big)
\]
and hence we obtain the claim of the proposition.
In the interest of these calculations, it is useful to note that the lower bounds on $\log Q^D$ in terms of $\kappa$ and $\tau$
that we assumed in the proposition implies that
\[
\max(\log d_0, \log d_1)\le C\log\log Q^D.
\]

\subsection{The case \texorpdfstring{$\a=2$}{a=2}}\label{sc:a2}

In this section, we consider the special case $V=\F_{p_1}\oplus\ldots\oplus\F_{p_M}$ and $\a_{i,1}=2$ for all $i$.
We write $\nu_2^{(d_1,d_2)}$ for the measure $\nu_\a^{(d_1,d_2)}$ with the above choice of $V$ and $\a$.
We will use this case later to estimate the probability that $P(2)$ is a proper power for a random polynomial,
which, in turn, yields an estimate for the probability that $P$ is a proper power of a polynomial.

Our main result is the following.

\begin{proposition}\label{pr:a2}
Let $\tau>0$ and assume $\|\mu\|_2^2\le 1-\tau$.
Suppose further that $\supp\mu\subset(-p_i/2,p_i/2)$ for each $i=1,\ldots,M$.

There is an absolute constant $C>0$ such that
for all $x\in V$ and $d\ge \frac{1}{\tau} (C \log(Q))^2$, we have
\[
|\nu_2^{(d)}(x)-Q^{-1}|\le Q^{-10}.
\]
\end{proposition}

The study of this case goes back to Chung, Diaconis and Graham \cite{CDG-RW}, who obtained very precise estimates
for the mixing time, which are much better than the bound $\log(Q)^2$ implied by the above result.
However, our application requires strong bounds for the distance between $\nu_2^{(d)}$ and the uniform
distribution, which was not considered in \cite{CDG-RW}.
Nevertheless, our proof draws on the ideas of \cite{CDG-RW} heavily.

We begin with a lemma on the Fourier coefficients of $\nu_2^{(d)}$.
Its proof relies on Lemma \ref{lm:F-single} and on the elementary fact that a sequence of the form
$\widetilde{\beta},\wt{\beta\cdot 2},\ldots [\beta\cdot 2^{\lfloor\log_2 Q\rfloor}]^\sim$ cannot stay below $Q/4$.

\begin{lemma}\label{lm:a2}
Let $q=\lfloor\log_2(Q)\rfloor$.
Then for any $l \in [0, d-q)$ and $\beta\in V\backslash \{0\}$, we have
\[
|\wh{\nu_2}^{(l,l+q)}(\beta)|\le\exp\Big(-\frac{1-\|\mu\|_2^2}{16}\Big).
\]
\end{lemma}

\begin{proof}
By Lemma \ref{lm:F-single}, we have
\begin{align*}
|\wh{\nu_2}^{(l,l+q)}(\beta)|
=&\prod_{j=1}^{q}|\wh{\nu_2}^{(l+j-1,l+j)}(\beta)|\\
\le&\exp\Big(-\sum_{n=l+1}^{l+q}\sum_{a_1,a_2\in\Z}\mu(a_1)\mu(a_2)R(a_1,a_2,n)^2/Q^2\Big),
\end{align*}
where
\[
R(a_1,a_2,n)=\Big[\Psi((a_1-a_2)2^n\beta)\Big]^\sim.
\]

We note that $R(a_1,a_2,n+1)\equiv 2R(a_1,a_2,n)\mod Q$.
Therefore, if $|R(a_1,a_2,n)|<Q/4$, then $|R(a_1,a_2,n+1)|=2|R(a_1,a_2,n)|$.
Now, it is easy to see that for any $a_1\neq a_2$, there is $n\in[l+1,l+q]$ such that
$|R(a_1,a_2,n)|\ge Q/4$, and the claim follows.
\end{proof}

\begin{proof}[Proof of Proposition \ref{pr:a2}]
We note that
\[
\nu_2^{(d)}(x)-Q^{-1}=
\frac{1}{|Q|}\sum_{\beta\in V\backslash\{ 0\}}\wh\nu_2^{(d)}(\beta) \chi_\beta(x),
\]
hence it is enough to prove that for all $\beta\in V\backslash \{0\}$,
\[
|\wh\nu_2^{(d)}(\beta)|<Q^{-10}.
\]

To that end, we choose an integer $L\le d/q$ and write
\[
\nu_2^{(d)}=\nu_2^{(-1,0)}*\nu_2^{(0,q)}*\ldots *\nu_2^{((L-1)q,Lq)}*\nu_2^{(Lq,d)}
\]
and note that Lemma \ref{lm:a2} implies
\[
|\wh\nu_2^{(d)}(\beta)|<\exp(-\tau L/16).
\]
This yields the desired estimate, if we set $L=\lceil160\log(Q)/\tau\rceil$,
which is permitted if the constant $C$ is taken sufficiently large.
\end{proof}

\section{Expected number of roots of a random polynomial}\label{sc:exp-root-2}

In this section, we use the results of the previous section to calculate the expected number of
roots of a typical polynomial in $\F_p$ for a random prime.
In the proofs of our main result, we will compare these with the formulae in Section \ref{sc:roots-random-fields}.

Let $m, d\ge 1$, $\kappa \in (0,\frac{1}{100})$ and $X>10$. For a random polynomial $P\in \Z[x]$ of degree at most $d$, we will now estimate  the number $B_P(p)$ of admissible roots of $P$ in $\F_p$ on average over the prime $p$. Here and below a residue modulo $p$ will be called admissible if it is $(\frac{\kappa}{m},mX)$-admissible in the notation of Definition \ref{admissible-res}. For the irreducibility results it will be sufficient to set $m=1$, but for information on the Galois groups, we will need to consider larger values of $m$. Nevertheless $m$ will not exceed a fixed power of $\log d$.

We suppose that the coefficients of $P$, except for the leading coefficient and the constant term, are
identically distributed and write $\mu$ for their common law. The notation $\E_P$ is used to denote expectation with respect to the law of the random polynomial $P$. Our purpose in this section is the prove the following result.

\begin{proposition}\label{pr:exp-root-2}
There are absolute constants $c_0,C_0>0$ such that the following holds. Let $\tau,\kappa>0$, $d,m\in\Z_{>0}$ and let $\mu$ be a probability measure on $\Z$ supported on  $[-\exp(d^{1/10}),\exp(d^{1/10})]$. Assume $\|\mu\|_2^2< 1- \tau$. Let $X$ be a number such that
\begin{equation}\label{condiX}
100 m^2 \max\{\kappa^{-1},\tau^{-1}, d^{1/10}\}<X<\frac{\kappa \tau}{C_0} \frac{d}{(m\log(md))^3}
\end{equation}
and let $g:\R\to\R_{\ge 0}$ be a function such that $\supp g\subset[X/2,X]$,
$g(x)\le2\exp(-x)$ for all $x$.
Then setting 
\[
Z:=\sum_{p}B_P(p)^m\log (p)g(\log p) - B_mw,
\]
we have
\[
\E_P(Z^2)\le \frac{1}{4}\Big(\exp(-\frac{X}{6}) + \exp(-\frac{c_0\tau\kappa d}{mX(\log(md))^{2}})\Big)
\]
where $B_m$ stands for the $m$-th Bell number and $w=\sum_{p}\log(p) g(\log p)$.
\end{proposition}

Recall that the Bell number $B_m$ is the number of equivalence classes on a set with $m$ elements.

\begin{corollary}\label{cor:exp-root-2} Under the assumption of the previous proposition, with probability at least
\[
1-\exp(-\frac{X}{6}) - \exp(-\frac{c_0\tau\kappa d}{mX(\log(md))^{2}})
\]
the following holds for $P$:
\[
|\sum_{p}B_P(p)^m\log (p)g(\log p) - B_mw| < \frac{1}{2}.
\]
\end{corollary}

\begin{proof} This is immediate from the last proposition after applying Chebyshev's inequality  
\[
\P_P(|Z|\ge \frac{1}{2}) \le 4\E_P(Z^2).
\]
\end{proof}

 We now pass to the proof of Proposition \ref{pr:exp-root-2} and begin by recording the following consequence of Proposition \ref{pr:equi-dist}.

\begin{lemma}\label{lm:exp-B} There are absolute constants $c_0,C_0>0$ such that the following holds. 
Let $\tau,\kappa>0$ and $m,d\in\Z_{>0}$.
Suppose that the probability measure $\mu$ on $\Z$ is supported on $[-\exp(d^{1/10}),\exp(d^{1/10})]$ 
and that $\|\mu\|_2^2<1-\tau$.
Let $X$ be such that
\[
10 \max\{\kappa^{-1},\tau^{-1},d^{1/10}\}<X<\frac{\kappa \tau}{C_0} \frac{d}{(m \log(md))^3},
\]
and let $p,p_1\neq p_2\in[\exp(X/2),\exp(X)]$ be primes.
Then
\begin{align*}
|\E_P[B_P(p)^m]-B_m|\le& B_m \cdot \Err(X,d,m),\\
|\E_P[B_P(p_1)^mB_P(p_2)^m]-B_m^2|\leq& B_m^2 \cdot \Err(X,d,m),
\end{align*}
where $\Err(X,d,m)=40m^2 \exp(-\frac{X}{5}) + \exp(-\frac{c_0 \tau \kappa d}{mX(\log(md))^{2}})$.
\end{lemma}

\begin{proof}
We write $A_p$ for the set of $(\frac{\kappa}{m},mX)$-admissible elements of $\F_p$.
In the notation of Section \ref{sc:RW} we take $M=1$ and $V=\F_p^m$.
Then
\[
\E[B_P(p)^m]=\sum_{\a\in (A_p)^m}\nu_\a^{(d)}(0).
\]

We decompose $(A_p)^m$ as a disjoint union of subsets $(A_p)^m(\eps)$ for which  Proposition \ref{pr:equi-dist} applies.
To this end, we write $\cE_m$ for the set of equivalence relations on the set $\{1,\ldots, m\}$.
For each $\e\in\cE_m$, we let $V(\e)$ be the subgroup of $V$ formed by the equations
$\a_i=\a_j$ whenever $(i,j)\in\e$, and write $(A_p)^m(\eps)$ for the subset of $(A_p)^m \cap V(\e)$ made of those $m$-tuples  $\a$ such that $\a_i=\a_j$ if and only if $(i,j)\in\e$.

Given $\e\in\cE_m$ we may apply Proposition \ref{pr:equi-dist} to the group $V(\e)\simeq \F_p^{m_\e}$, where $m_\e$ is the
number of equivalence classes in $\e$ and obtain
\begin{equation}\label{eq:1orbit}
\bigg|\sum_{\a\in (A_p)^m(\e)}\nu_\a^{(d)}(0)-\frac{|(A_p)^m(\e)|}{p^{m_\e}}\bigg|<\exp(-c_0 \frac{\tau \kappa}{mX}\frac{d}{(\log(md))^{2}}).
\end{equation}

As we have already noted, the number of polynomials of degree at most $10mX$ and Mahler measure at most $\exp(\kappa/m)$ is at most $\exp(X/10)$ by \cite{DK-Lehmer}*{Theorem 1}.
Therefore, $|\F_p \setminus A_p|\leq 10mX \exp(X/10)$ and 
\[
0\leq 1 - \frac{|(A_p)^m(\e)|}{p^{m_\e}}\leq 10m^2 X \exp(X/10)/p \leq 20m^2  \exp(-X/5).
\]
Now summing up \eqref{eq:1orbit} for $\e\in\cE_m$, we arrive at the first claim.

The proof of the second claim is entirely similar using Proposition \ref{pr:equi-dist} for the random
walk on $V=\F_{p_1}^m\oplus\F_{p_2}^m$.
We leave the details to the reader.
\end{proof}

\begin{lemma}\label{lm:diag}
Let $X>10$ and let $g:\R\to\R_{\ge 0}$ be a function such that $\supp g\subset[X/2,X]$,
$g(x)\le2\exp(-x)$ for all $x$. Then
\[
w_2:=\sum_{p}(\log (p)g(\log p))^2\le 8 X^2\exp(-X/2)\]
\[
w:=\sum_{p}\log (p)g(\log p) \le 4X^2.
\]
\end{lemma}
\begin{proof}
A simple calculation yields
\[
\sum_{p}(\log (p)g(\log p))^2
\le\sum_{\exp(X/2)\le n \le \exp(X)}4\frac{(\log n)^2}{n^2}
\le 8 X^2\exp(-X/2),\]
\[\sum_{p} \log (p)g(\log p)
\le\sum_{\exp(X/2)\le n \le \exp(X)}2\frac{\log n }{n}
\le 4 X^2.
\]
\end{proof}

\begin{proof}[Proof of Proposition \ref{pr:exp-root-2}]
Recall that
\[
Z=\sum_{p}B_P(p)^m\log (p)g(\log p) - B_mw.
\]
Setting $h(x)=\log(x) g(\log x)$ we compute:
\[
Z^2=\sum_{p_1, p_2} (B_P(p_1)^m-B_m)(B_P(p_2)^m-B_m)h(p_1)h(p_2)
\]
so
\begin{align*}
\E_P(Z^2)
=&\sum_{p_1,p_2}\E_P((B_P(p_1)^m-B_m)(B_P(p_2)^m-B_m))h(p_1)h(p_2)\\
=&\sum_{p_1,p_2}\E_P(B_P(p_1)^mB_P(p_2)^m)h(p_1)h(p_2)\\
&-2B_m\sum_{p_1}h_{p_1}\sum_{p_2}\E_P(B_P(p_2)^m)h(p_2)
+B_m^2\Big(\sum_p h(p)\Big)^2\\
=&\sum_{p_1,p_2}\E_P(B_P(p_1)^mB_P(p_2)^m-B_m^2)h(p_1)h(p_2)\\
&-2B_m\sum_{p_1}h(p_1)\sum_{p_2}\E_P(B_P(p_2)^m-B_m)h(p_2)\\
=&\sum_{p_1 \neq p_2}\E_P(B_P(p_1)^mB_P(p_2)^m - B_m^2) h(p_1)h(p_2)\\
&-2B_mw\sum_p \E_P(B_P(p)^m - B_m)h(p) +\sum_p \E_P(B_P(p)^{2m}-B_m^2)h(p)^2. 
\end{align*}
We use Lemma \ref{lm:exp-B} to bound the first two terms and the crude bound $B_P(p)\le d$ for the  third:
\[
\E_P(Z^2)\le 3B_m^2\Err(X,d,m)w^2+d^{2m}w_2.
\]
We recall that $100m^2d^{1/10}<X$, so that $d^{2m}\leq e^{X/50}$. 
We plug in the bounds for $w$ and $w_2$
from Lemma \ref{lm:diag} and the definition $\Err(X,d,m)$ from Lemma \ref{lm:exp-B} 
and obtain
\begin{align*}
\E_P(Z^2)\le&48 B_m^2X^4\Big(40m^2\exp(-\frac{X}{5})+\exp(-\frac{c_0\tau\kappa d}{mX(\log(md))^2})\Big)\\
&+8X^2\exp(-\frac{X}{2}+\frac{X}{50}).
\end{align*} 
The constraints on $X$ in the statement of the proposition imply that
\[
\frac{c_0\tau\kappa d}{mX(\log(md))^2}>C m\log(md),
\]
where $C$ is an arbitrarily large number provided we set $C_0$ sufficiently large (depending on $c_0$).
This means that we can absorb the factor $48B_m^2 X^2$ into the constant $c_0$.
Similarly, the lower bound on $X$ implies that we can also absorb the factor $48 B_m^2X^4 \cdot 40m^2$
at the expense of replacing $\exp(-X/5)$ by $\exp(-X/6)$.
\end{proof}

\section{Polynomials of small Mahler measure}\label{sc:Lehmer}

In this section, we estimate the probability that the random polynomial $P$ is divisible by
a non-cyclotomic polynomial of small Mahler measure.
The following result and the ideas in its proof are inspired by Konyagin's paper \cite{Kon-01}.

\begin{proposition}\label{pr:small-Mahler}
Let $P=A_dx^d+\ldots+A_1x+A_0\in\Z[x]$ be a random polynomial with independent
coefficients, and write $\mu_j$ for the law of $A_j$.
Let $\tau>0$ be a number.
We assume
\[
\supp\mu_j\subset[-\exp(d^{1/10}),\exp(d^{1/10})]
\]
for all $j$ and $\|\mu_j\|_2^2\le1-\tau$ for all $j\neq0,d$.

Then the probability that there is a non-cyclotomic polynomial $Q$ with $\log M(Q)<\tau/10$
dividing $P$  is at most $2\exp(-c \tau d^{4/5})$, where $c>0$ is an absolute constant.
\end{proposition}

The exponent $4/5$ is not optimal and there is a trade-off between it and the bound imposed on the
coefficients of $P$.
Since any improvement of this bound would have no effect on our theorems, we leave it to the interested
reader to find the optimal bound that can be derived from the proof.

We give two simple Lemmata that estimate the probability that a fixed single polynomial $Q$
divides a random polynomial.
Both of them are implicitly contained in \cite{Kon-01}.
The first one is useful when $\deg{Q}$ is large.

\begin{lemma}\label{lm:indi1}
Let $P=A_dx^d+\ldots+A_1x+A_0\in\Z[x]$ be a random polynomial with independent
coefficients, and write $\mu_j$ for the law of $A_j$.
Let $Q\in\Z[x]$ be a polynomial of degree $n\le d$.

Then
\[
\P_P(Q|P)\le\|\mu_0\|_\infty\cdots\|\mu_{n-1}\|_\infty.
\]
\end{lemma}
\begin{proof}
Write $R$ for the remainder of $A_dx^d+\ldots+A_nx^n$ modulo $Q$ in $\Q[x]$.
If $Q|P$, then $R=-A_{n-1}x^{n-1}-\ldots-A_0$.
Therefore, the probability of $Q|P$ conditioned on the value of $A_dx^d+\ldots+A_nx^n$ is bounded by the
maximal probability of $A_0,\ldots,A_{n-1}$ taking any given value, which is precisely
the claimed bound.
\end{proof}

\begin{lemma}\label{lm:indi2}
Let $P=A_dx^d+\ldots+A_1x+A_0\in\Z[x]$ be a random polynomial with independent
coefficients, and write $\mu_j$ for the law of $a_j$.
Let $H\in\Z_{>0}$ and $\tau>0$ be numbers.
We assume $\supp\mu_j\subset[-H,H]$ for all $j$ and $\|\mu_j\|_2^2\le1-\tau$ for all $j\neq0,d$.
Let $Q\in\Z[x]$ be a non-cyclotomic irreducible polynomial.

Then for $d$ larger than some absolute constant:
\[
\P(Q|P)\le \exp(-c\tau d(\log H+\log d)^{-1}(\log d)^{-3}),
\]
where $c>0$ is some absolute constant.
\end{lemma}
\begin{proof}
Let
\[
s=\frac{\log(2H(d+1)^{1/2})}{c(\log\log d)^3/(\log d)^3},
\]
where $c$ is a sufficiently small constant so that
\[
\log(M(Q))>c(\log\log d)^3/(\log d)^3.
\]
The existence of such a constant follows by Dobrowolski's bound \cite{Dob-Lehmer}.

By Lemma \ref{lm:q}, there is a prime $q \in (s,2s]$ such that the ratio of any two roots
of $Q$ is not a root of unity of order $q$.
Let $P_1,P_2\in\Z[x]$ be two polynomials with coefficients of absolute value at most $H$ that differ only in some of the coefficients of monomials
of the form $x^{qj}$ for $j\in\Z_{\ge 0}$.
If $Q|(P_1-P_2)$, then each number $z\omega$ is a root for $P_1-P_2$, where $z$ is a root of
$Q$ and $\omega$ is a $q$-th root of unity. And by our choice of $q$ all $z\omega$ are distinct as $z$ ranges over the roots of $Q$ and $\omega$ over the $q$-th roots of unity.
This implies that
\[
M(P_1-P_2)\ge M(Q)^q> 2H(d+1)^{1/2},
\]
which is impossible by \eqref{measure-bounds}. This means that for any given choice of integers $b_j \in [-H,H]$ for those $j\leq d$ that are not a multiple of $q$, in each class of $\Z[x]$ modulo $Q$ there is at most one polynomial $P=a_0+\ldots+a_dx^d$ with $a_j=b_j$ for all such $j$.

Hence conditioning on the value of $a_j$ for all indices $j$ that are not multiples of $q$, the probability of $Q|P$ is bounded by the probability that the rest of the coefficients take any particular given value.
Therefore
\[
\P(Q|P)\le\|\mu\|_\infty^{\lfloor d/q\rfloor-1}.
\]
\end{proof}

\begin{proof}[Proof of Proposition \ref{pr:small-Mahler}]
We fix a small number $\e>0$.
Let $j\ge 0$ be an integer and write $\cQ_j$ for the set of non-cyclotomic irreducible polynomials $Q$ with 
$\deg Q=j$ and  $\log M(Q)<\tau/10$. 
By the estimate of Dubickas and Konyagin \cite{DK-Lehmer}*{Theorem 1}, we have
$|\cQ_j|\le\exp(\tau j/10)$ if $j$ is sufficiently large.

Using Lemma \ref{lm:indi1}, we then have
\[
\P(\exists Q \in\cQ_j:Q|P)\le\exp(\tau j/10)\cdot\exp(-\tau j/2)\le\exp(-\tau j/10)
\]
for each $j$.
By Lemma \ref{lm:indi2} applied with $H=\exp(d^{1/10})$
\[
\P(\exists Q \in\bigcup_{j<d^{4/5}}\cQ_j:Q|P)\le\exp(\tau d^{4/5}/10)\cdot \exp(-\tau d^{4/5})\le\exp(-\tau d^{4/5}/10).
\]
provided $d$ is sufficiently large depending on an absolute constant.

Summing up the above bounds we get
\[
\P(\exists Q \in\bigcup\cQ_j:Q|P)\le\exp(-\tau d^{4/5}/10)+\sum_{j\ge d^{4/5}}\exp(-\tau j/10),
\]
which proves the claim.
\end{proof}

\section{Proper powers}\label{sc:powers}

In this section, we estimate the probability that a random polynomial $P$ is of the form
$\Phi Q^k$ with $k>1$, where $\Phi$ is the product of cyclotomic factors.

\begin{proposition}\label{pr:powers}
Let $P=A_dx^d+\ldots+A_1x+A_0\in\Z[x]$ be a random polynomial with independent coefficients.
Assume that $A_1,\ldots, A_{d-1}$ are identically distributed with common law $\mu$.
Assume further that all coefficients are bounded by $\exp(d^{1/10})$ almost surely.
Let $\tau>0$ be a number such that $\|\mu\|_2^2<1-\tau$.

Then there are absolute constants  $c,C>0$ such that the probability that
$P=\Phi Q^k$, where $\Phi$ is a product of cyclotomic polynomials, $Q\in\Z[x]$ and $k\ge 2$,
is less than $2\exp(-c (\tau d)^{1/2})$, provided $d$ is larger than $C/\tau^4$.
\end{proposition}

In the next two lemmas we keep the assumptions of Proposition \ref{pr:powers}. The first is a reformulation of Proposition \ref{pr:a2}.

\begin{lemma}\label{lm:P2}
There is an absolute constant $c_0>0$ such that the following holds.
Let $q<\exp(c_0 (\tau d)^{1/2})$ be a product of distinct primes larger than $2\exp(d^{1/10})$.
Then for every $a\in\Z$, we have
\[
|\P_P[P(2)\equiv a \mod q]-q^{-1}|<q^{-10}.
\]
\end{lemma}

\begin{lemma}\label{lm:powers}
Fix $R\in\Z[x]$,
and fix an integer $2\le k\le d^{1/5}$.
Then
\[
\P_P[P=RQ^k\text{ for some $Q\in\Z[x]$}]\le \exp(-c (\tau d)^{1/2}),
\]
where $c>0$ is an absolute constant.
\end{lemma}

In the proof that follows, we will use the upper bound on $k$ in only one place, where
we apply the prime number theorem in arithmetic progressions.
It would be sufficient to impose a significantly milder upper bound on $k$, but
we will see that $P=RQ^{k}$ may hold with $k>d^{1/5}$ only if $Q$ is cyclotomic.

\begin{proof}[Proof of Lemma \ref{lm:powers}]
If $R(2)=0$ and $R$ divides $P$, then $P(2)=0$. Picking a prime $q$ in the interval $(\frac{1}{2}\exp(c_0 (\tau d)^{1/2})/2,\exp(c_0 (\tau d)^{1/2}))$, Lemma \ref{lm:P2} implies that
\[
\P_P[P(2)\equiv 0\mod q]<2/q.
\]
So we can safely assume in the rest of the proof that $R(2)\neq 0$.
We note also that $|R(2)|\le |P(2)|\le \exp(d^{1/10})2^{d+1}$.

We denote by $\cP$
the collection of primes
\[
p\in[\frac{1}{2}\exp(c_0 (\tau d)^{1/2}/2)/2,\exp(c_0 (\tau d)^{1/2}/2)]
\]
such that $p\nmid R(2)$
and $k|p-1$.
It follows from the prime number theorem in arithmetic progressions \cite{davenport}*{Chp. 20, (10)} that there are more than
\[
|\cP|\ge \exp(c_0 (\tau d)^{1/2}/4)
\]
such primes if $d$ is sufficiently large (i.e. $\tau d$ larger than an effective constant: we are counting primes between $x/2$ and $x$ that are congruent to $1$ modulo $k$ with $k$ allowed to take any value $\ll (\log x)^{2/5}$ say).

For each $p\in\cP$ we denote by $X_p$ the random variable that is equal to $1$ if
\[
P(2)\equiv R(2)a^k\mod p
\]
for some $a\in\Z/p\Z$ and that is equal to $0$ otherwise.
If $P=RQ^k$ for some $Q\in\Z[x]$, then clearly $X_p=1$ for all $p\in\cP$.

It follows from Lemma \ref{lm:P2} applied first to $q=p_1$ and then to $q=p_1p_2$ that
\begin{align*}
\E_P[X_{p_1}]=\E_P[X_{p_1}^2]=\frac{(p_1-1)/k+1}{p_1}+O(\exp(-9c_0(\tau d)^{1/2}/2))\\
\E_P[X_{p_1}X_{p_2}]=\frac{(p_1-1)/k+1}{p_1}\cdot\frac{(p_2-1)/k+1}{p_2}+O(\exp(-9c_0 (\tau d)^{1/2}/2))
\end{align*}
for any $p_1\neq p_2\in\cP$.
Therefore, writing $Y=\sum_{p\in\cP} X_p$, since $k \ge 2$, 
\[
\E_P Y\le\frac{2}{3}|\cP|
\]
and the variance $\mathbb{V}ar(Y)= \E_P Y^2 - (\E_P Y)^2$ is bounded by
\begin{align*}
\mathbb{V}ar(Y)
=&\sum_{p\in\cP}\Big(\frac{(p-1)/k+1}{p}-\Big(\frac{(p-1)/k+1}{p}\Big)^2\Big)\\
&+O(|\cP|^2\exp(-9c_0 (\tau d)^{1/2}/2))\\
\leq& \frac{2}{3}|\cP| +1 <|\cP|
\end{align*}
provided $d$ is sufficiently large.
We conclude from Chebyshev's inequality that
\[
\P_P(Y=|\cP|)\leq \P_P(Y-\E_P Y \ge \frac{1}{3}|\cP|) \leq \mathbb{V}ar(Y) \big(\frac{3}{|\cP|}\big)^2  < \frac{9}{|\cP|},
\]
which proves the lemma.
\end{proof}

\begin{proof}[Proof of Proposition \ref{pr:powers}]
Boyd and Montgomery \cite{BM-cyclotomic} gave an asymptotic formula for the number of polynomials $\Phi$
in $\Z[x]$ of degree $n$ that are the product of their cyclotomic factors.
In particular, they proved that there are at most $\exp(C_0n^{1/2})$ such polynomials,
where $C_0$ is an absolute constant ($C_0=4$ works for large enough $n$).

For a fixed $\Phi$ we may apply Lemma \ref{lm:indi1} and conclude that the probability that $\Phi$ divides $P$ is at most
$\exp(- \tau \deg(\Phi)/2)$. Therefore, the probability that $P=\Phi Q^k$ for some $\Phi$ with
\[
\deg\Phi\ge 4\frac{C_0}{\tau}d^{1/2}
\]
is at most $\exp(-C_0 d^{1/2})$.

We consider now the probability that $P=\Phi Q^k$ with $\Phi$ having smaller degree.
We can assume that $Q$ is not a product of cyclotomic factors, otherwise it can be absorbed into $\Phi$
and it is covered by the previous case.
We note that if $P=\Phi Q^k$, then by \eqref{measure-bounds}
\[
M(Q)=M(P)^{1/k}\le\exp(d^{1/10}/k)(d+1)^{1/2k}.
\]
Since $Q$ is not a product of cyclotomic factors, this implies that $k\le d^{1/5}$ (say) by Dobrowolski's bound \eqref{measure-bounds}.

Again by \cite{BM-cyclotomic} the number of polynomials in the role of $\Phi$ that are not covered by the previous case is at most
\[
\exp\big(2C_0(C_0/\tau)^{1/2}d^{1/4}\big).
\]
Now we can use Lemma \ref{lm:powers} to estimate the probability of $P=\Phi Q^k$ for individual
choices of $\Phi$ and $k$ and conclude the proof.
\end{proof}

\section{Proof of the main results}\label{sc:proofs}

We first give a simple lemma that allows us to decide when a permutation group is $m$-transitive. Recall that the Bell number $B_m$ is the number of equivalence relations on a set with $m$ elements.

\begin{lemma}\label{lm:perm}
Let $G$ be a permutation group acting on a set $\Omega$ and let $m\in\Z_{>0}$.
Suppose $|\Omega|\ge m$. The number $|\Omega^m/G|$ of orbits of $G$ acting diagonally on $\Omega^m$ satisfies 
\[
|\Omega^m/G| \ge B_m
\]
with equality if and only if the action of $G$ on $\Omega$ is $m$-transitive.
\end{lemma}

\begin{proof}
If $G$ is $m$-transitive, then its orbits on $\Omega^m$ are in one-to-one
correspondence with equivalence relations on the set of coordinates.
Given an equivalence relation on the $m$ coordinates, the corresponding orbit is the set of tuples in $\Omega^m$ whose coordinates are equal if and only if they are related by the equivalence relation.
Since $|\Omega|\ge m$, all equivalence relations can occur. Hence $|\Omega^m/G|=B_m$.

Now in the general case $G \le \Sym(\Omega)$, so each orbit of $G$ is contained in an orbit of $\Sym(\Omega)$. Thus  $|\Omega^m/G| \ge |\Omega^m/\Sym(\Omega)| = B_m$. 

If $G$ is not $m$-transitive, then the orbit of the full symmetric group $\Sym(\Omega)$ consisting of tuples with distinct coordinates
splits into multiple orbits of $G$, hence $|\Omega^m/G|>B_m$.
\end{proof}

\subsection{Proof of Theorem \ref{th:GRH}}\label{sc:proofGRH}

We set $\kappa=\tau/100$, $m=1$ and let $X>10$. Recall that we denote by $\wt P$ the product of the $(X,\kappa)$-admissible irreducible factors of $P$ and that $\Omega$ is the set of complex roots of $\wt P$ (see Definition \ref{admissibility}). We aim to show that the Galois group $G$ of the splitting field of $\wt P$ acts transitively on $\Omega$
with high probability.

Recall that $h_X$ is the function $h_X(u)=2e^{-X}1_{(X-\log 2, X]}(u)$. It follows from the prime number theorem that
\[
w:=\sum_{p}\log(p) h_X(\log p)\to 1
\]
as $X\to \infty$. We apply Corollary \ref{cor:exp-root-2} for $g=h_X$ with $m=1$. It applies if $X$ is large enough and we conclude that
\begin{equation}\label{eq:good-poly}
|\sum_{p}B_P(p)\log (p)h_X(\log p) - 1| \leq \frac{2}{3}
\end{equation}
holds for any $X\in[100d^{1/10},\frac{\tau^2}{100C_0}(\log d)^{-3}d]$ with probability at least
\[
1-\exp(-\frac{X}{6}) - \exp(-\frac{c_0\tau^2 d}{100X(\log(d))^{2}})
\]
provided $d>100/\tau$ say. Taking $X=\tau (c_0 d/100)^{1/2}/\log d$ (which is allowed provided $d\tau^4$
is sufficiently large) this bound becomes $\ge 1-2\exp(-X/6)$. We now  assume that
\eqref{eq:good-poly} holds for $P$, and $\zeta_K$ satisfies RH for all $K=\Q(a)$ for any root $a$ of $P$.
By Proposition \ref{pr:Bp-sum-GRH}, we then have
\[
\sum_{p} B_P(p)\log(p)h_X(\log p )=|\Omega/G|+O(\exp(-X/10)).
\]
If $d$ is sufficiently large, we can conclude that
\[
|1-|\Omega/G||<1
\]
under the above assumptions on $P$. We therefore  conclude that  $|\Omega/G|=1$, and hence $G$ acts transitively on $\Omega$, i.e. $\wt P$ is irreducible.

By Proposition \ref{pr:small-Mahler}, with probability at least $1-2\exp(-c \tau d^{4/5})$,
any exceptional factor of $P$ is cyclotomic.
If that holds in addition to the hypothesis we have already made, then $P=\Phi\wt P^k$,
where $\Phi\in\Z[x]$ is a product of
a power of $x$ and cyclotomic polynomials, and $k\in\Z_{>0}$.

By Proposition \ref{pr:powers}, we know that $k=1$ with probability at least $1-2\exp(-c(\tau d)^{1/2})$. Furthermore, the probability that $\deg(\Phi) \ge \frac{C}{\tau} \sqrt{d}$ is at most $\exp(-C d^{1/2}/4)$, because this is true for any given polynomial $\Phi$ by Lemma \ref{lm:indi1} and, as recalled in the proof of Proposition  \ref{pr:powers}, there are at most $\exp(C d^{1/2}/4))$ such poynomials for some absolute constant $C>0$.
This establishes part $(1)$ of Theorem \ref{th:GRH}. 

The proof of part $(2)$ is similar, but we need to also consider moments of $B_P(p)$ of order $m>1$ in order to show that $|\Omega^m/G|=B_m$ and hence conclude, by  Lemma \ref{lm:perm}, that $G$ acts $m$-transitively on $\Omega$. An old fact, going back to  Bochert and Jordan \cite{jordan1895} in the 19-th century, asserts that every degree $d$ permutation group that is at least $(30\log d)^2$-transitive must contain the alternating group $\Alt(d)$. A simple proof of a slightly better bound can be found in \cite{babai-seress} (see also \cite{dixon-mortimer}*{Theorem 5.5.B} where Wielandt's stronger bound $6 \log d$ is proved). Using the classification of finite simple groups it is now known that there is a bound independent of $d$ and indeed every $6$-transitive group contains $\Alt(d)$ (see  \cite{cameron}*{Corollary 5.4}). But we choose not to rely on the classification, since, at the expense of loosing a $\log(d)$ factor in the probability of exceptions, we can avoid it. In fact if instead we use Wielandt's bound (whose proof is more involved) we can get the slightly better bound $\exp(-c \tau d^{1/2}/(\log d)^{3/2})$ in  (2) of Theorem \ref{th:GRH}.

So let $m\ge1$, $\kappa=\tau/100$ and $X>10$ and consider $\wt P$ the product of the irreducible $(\frac{\kappa}{m},mX)$-admissible factors of $P$ and as earlier $B_P(p)$ the number of  $(\frac{\kappa}{m},mX)$-admissible roots of $P$ in $\F_p$. 

By Corollary \ref{cor:exp-root-2} applied to $g=h_X$ we get that 
\begin{equation}\label{m-bd}|\sum_p B_P(p)^m\log(p)h_X(\log p) - B_m w|<\frac{1}{2}\end{equation}
with probability at least
\[
1-\exp(-\frac{X}{6}) - \exp(-\frac{c_0\tau^2 d}{100mX(\log(md))^{2}})
\]
provided $X$ is in the  interval allowed by \eqref{condiX}.  We now set $m=\lceil 30(\log d)^2 \rceil$ and $X^2=c_0\tau^2/100 \cdot d/(m (\log(md))^2)$. Note then that when $\tau^4 d$ is large enough $X$ is in the allowed interval and that \eqref{m-bd} holds with probability at least $1-2\exp(-X/6)$. Assume now that \eqref{m-bd}  holds for $P$ and that $\zeta_K$ satisfies RH for all $K=\Q(a_1,\ldots,a_m)$ for any choice of $m$ roots of $P$. By Proposition \ref{pr:Bp-sum-GRH} we then get 
\[
\sum_{p} B_P(p)^m\log(p)h_X(\log p )=|\Omega^m/G|+O(\exp(-X/10)).
\]
If $d$ is large enough $|w-1|= O(X^2\exp(-X/2))$ by Proposition \ref{pr:prime-sum-GRH}  (assuming RH for $\zeta_\Q$). Since $B_m \leq 2^{m^2} \le \exp(X/100)$ this implies that 
\[
|B_m - |\Omega^m/G||\leq \frac{1}{2} + B_m|1-w| + O(\exp(-X/10)) <1
\]
as soon as $d$ is large enough, and hence that $|\Omega^m/G|=B_m$. So by Lemma \ref{lm:perm} $G$ acts $m$-transitively on $\Omega$ and by the 19-th century transitivity bound recalled earlier, since $\deg \wt P \leq d$, $G$ contains the alternating group $\Alt(\deg \wt P)$. Finally as in part (1), except for a small set of exceptions $P=\Phi \wt P$, and this completes the proof of the theorem.

\subsection{Proof of Corollaries  \ref{cr:00} and \ref{cr:01}}

The following lemma is implicitly contained in \cite{Kon-01}*{pp. 345}
\begin{lemma}\label{lm:cycl-div}
Let $\omega_n$ be the $n$-th cyclotomic polynomial of degree $\f(n)$.
Then for all $n,d$,
\[
\P[\omega_n|P_d]\le \big(C(\mu)\frac{n}{d}\big)^{\f(n)/2},
\]
where $C(\mu)>0$ depends only on $\mu$.
\end{lemma}

\begin{proof}Write $Q_d=\sum_{j=0}^{n-1} B_j x^j$, where $B_j:=\sum_{i \equiv j \mod n, 0\le i\le d} A_i$. Note that if $\omega_n |P_d$, then $\omega_n|Q_d$, and hence Lemma \ref{lm:indi1} implies that 
\[
\P[\omega_n|P_d]\le  \prod_{0}^{\f(n)-1} \|\mu_j\|_{\infty},
\]
where $\mu_j$ is the law of $B_j$, which is the sum of roughly $\lfloor d/n \rfloor$ i.i.d. variables with common law $\mu$. Since $\mu$ has a finite second moment, there is a constant $C(\mu)>0$ such that we have $\|\mu_j\|_\infty \leq (C(\mu) n/d)^{1/2}$, as follows say from the local limit theorem. The claim follows.
\end{proof}

We apply this lemma for different ranges of $n$. If $N\le \f(n)\le  100N$, then $n$ is bounded in terms of $N$ and
\[
\P[\omega_n|P]\le (C(\mu) n/d)^{N/2} = O_{N,\mu}(d^{-N/2})
\]
If $100N< \f(n) \le d^{1/2}$, then $n\leq C d^{1/2}\log \log d$ for some absolute constant $c>0$ and 
\[
\P[\omega_n|P]\le (C(\mu) c\log \log d/d^{1/2})^{50N}
\]
If $d^{1/2} \le \f(n) \le d$, then
\[
\P[\omega_n|P]\le \|\mu\|_\infty^{d^{1/2}/2}
\]
by Lemma \ref{lm:indi1}. Summing over all such $n$'s we get:
\[
\P[\omega_n|P \textnormal{ for some }n\textnormal{ with }\f(n)\ge N] = O_{\mu,N}(d^{-N/2}).
\]

In order to apply Theorem \ref{th:GRH}, we need to truncate the coefficients. But
\[
\P[ \max_{0\le i\le d} |A_i|>e^{d^{1/10}}] \leq (d+1) \P[|A_0|>e^{d^{1/10}}] \leq (d+1) e^{-2d^{1/10}}\E[|A_0|^2] 
\]
by Chebyshev's inequality. The proof of Corollary \ref{cr:00} now follows by combining the above inequalities with Theorem \ref{th:GRH}.

To get Corollary \ref{cr:01} take $N=2$ and observe that the law of $P$ in the statement is designed to make sure that $x\nmid P$ always. Note also that Lemma \ref{lm:cycl-div} still holds even though $A_0$ and $A_d$ are not distributed like the other $A_i$'s, so the above estimates continue to hold. Since $P$ has non-negative coefficients and at least two positive ones $P(1)>0$. So it is only left to estimate $\P[\omega_2|P]=\P[P(-1)=0]$. Looking at $P(-1)$ yields a random walk on $\Z$ and it is therefore a simple matter to verify that  $\P[P(-1)=0]=\sqrt{\frac{2}{\pi d}}+O(d^{-1})$, as desired.

\subsection{Proof of Theorem \ref{th:mild-hyp1}}

The proof is identical to that of part $(1)$ of Theorem \ref{th:GRH}, except that we take $X=d(\log d)^{-\b}$
and apply Proposition \ref{pr:Bp-sum} instead of Proposition \ref{pr:Bp-sum-GRH}.
We note that the exceptional zeros are not present by the assumptions of the theorem, so
the right hand side of the displayed formula in Proposition \ref{pr:Bp-sum} becomes
\[
|\Omega/G|+O(\exp(-c(\log d)^{\a-\b})).
\]

\subsection{Proof of Theorem \ref{th:mild-hyp2}}

Set $\b:=\a-\g$. As in the proof of Theorem \ref{th:GRH}, we set $\kappa=\tau/10$ for the admissibility parameter (see Def. \ref{admissibility}). By Proposition \ref{pr:small-Mahler} with probability at least $1-2\exp(-c_\tau d^{4/5})$ every non-cyclotomic irreducible factor of $P$ has Mahler measure at least $\exp(\kappa)$. We may thus assume that $P$ has this property, and let $\wt P$ be the product of the non-cyclotomic irreducible factors of $P$. As before $\Omega$ is the set of roots of $\wt P$ and $G$ the Galois group of the splitting field of $P$.

We use Proposition \ref{pr:exp-root-2} with $m=1$,  $X=X_1=2d(\log d)^{-\b}$
and $X=X_2=d(\log d)^{-\b}$ for the functions $g=g_{X_1,k}$ and $g=g_{X_2,k}$, respectively, where
$k=\lfloor(\log d)^{\a-\b}/10\rfloor$.
We can conclude that 
\[
\E_P(Z_i^2)=O(\exp(-c(\log d)^{\b-2}))
\]

holds for each $i=1,2$, where 
$$Z_i:=\sum_{p} B^i_P(p)\log(p)g_{X_i,k}(\log p)-w_i,$$ $w_i:=\sum_p\log (p) g_{X_i,k}(\log p)$
and $B^i_P(p)$ is the set of $(X_i,\kappa)$-admissible roots of $P$ in $\F_p$. Hence by Chebychev's inequality $\P_P(|Z_i|>t) \leq t^2 \E_P(Z_i^2)$, we obtain that  with probability at least $1-2\exp(-c(\log d)^{\b-2})$
\begin{equation}\label{eq:X1-sum}|Z_i|=O(\exp(-c(\log d)^{\b-2}))\end{equation}
holds for each $i=1,2$.

We note that $|1-w_i|<C\exp(-c(\log d)^{\a-\b})$ as can be seen for example from Proposition \ref{pr:prime-sum}
applied for $K=\Q$, since the Riemann zeta function $\zeta_\Q$ has no zeros in a sufficiently small neighborhood of $1$.
(Significantly better bounds can be obtained by the proof of Proposition \ref{pr:prime-sum}, but this is not needed.)

Now we assume that $P$ satisfies \eqref{eq:X1-sum} for both $i=1,2$. 
We apply
Proposition \ref{pr:Bp-sum} and obtain for $i=1,2$
\begin{equation}
Z_i+w_i=\sum_{O \in \Omega/G}(1-G_{X_i,k}(\rho_{K_O,0}))+O(\exp(-c(\log d)^{\a-\b}))\label{eq:X1-sum-2}.
\end{equation}

Now we combine the above estimates 
 and $|w_1-w_2|<C\exp(-c(\log d)^{\a-\b})$ to get
\[
\sum_{O \in \Omega/G}(G_{X_2,k}(\rho_{K_O,0})-G_{X_1,k}(\rho_{K_O,0}))\le C(\exp(-c(\log d)^{\a-\b})+\exp(-c(\log d)^{\b-2})).
\]
We note that
\begin{align*}
(G_{X_2,k}(\rho_{K_O,0})-G_{X_1,k}(\rho_{K_O,0}))
=&\Big(1-\frac{G_{X_1,k}(\rho_{K,0})}{G_{X_2,k}(\rho_{K_O,0})}\Big) G_{X_2,k}(\rho_{K_O,0})\\
\ge& (1-\exp(-(1-\rho_{K_O,0})(X_1-X_2)/4))G_{X_2,k}(\rho_{K_O,0})\\
\ge& c\exp(-c_0(\log d)^\gamma)d(\log d)^{-\b}G_{X_2,k}(\rho_{K_O,0}).
\end{align*}
Here we used the bound on $G_{X_1,k}/G_{X_2,k}$ from Lemma \ref{lm:weights} and then the assumption on
the exceptional zeros from the theorem, and the constant $c_0$ is the constant $c$ in that bound.
Therefore, we can conclude that
\[
\sum_{O \in \Omega/G}G_{X_2,k}(\rho_{K_O,0})\le C(\exp(-c(\log d)^{\a-\b})+\exp(-c(\log d)^{\b-2}))
\]
if we choose $c_0$ sufficiently small, since $\a-\b, \b-2\ge \g$ and $\g>1$.

We combine the last estimate with \eqref{eq:X1-sum}  and \eqref{eq:X1-sum-2} and we can write
\[
|\Omega/G|=1+O(\exp(-c(\log d)^{\a-\b})+\exp(-c(\log d)^{\b-2})),
\]
hence $|\Omega/G|=1$ as it is an integer. Therefore $\wt P$ is irreducible. Now we can finish the proof by applying Proposition \ref{pr:powers}.

\subsection{Proof of Theorem \ref{th:exp-zero}}

We pick a number $\a'\in(\b,\a)$.
We use Proposition \ref{pr:exp-root-2} with $m=1$,
$g=g_{X,k}$, where $X=d(\log d)^{-\b}$ and
$k=\lfloor(\log d)^{\a'-\b}/10\rfloor$.
After applying Chebychev's inequality as in the proof of Theorem \ref{th:mild-hyp2}, we get that
$$\sum_{p} B_P(p)\log(p)g_{X,k}(\log p)=w+O(\exp(-c(\log d)^{\b-2})),$$
holds with probability at least $1-C\exp(-c(\log d)^{\b-2})$,
where
\[
w=\sum_{p}\log(p)g_{X,k}(\log p).
\]

Moreover, as before, using Proposition \ref{pr:prime-sum} for the field of rational numbers we see that  $|w-1| = O(\exp(-c(\log d)^{\a - \b}))$. Hence
\begin{equation}\label{eq:BP-sum-exp-zero}
\sum_{p} B_P(p)\log(p)g_{X,k}(\log p)=1+O(\exp(-c(\log d)^{\min\{\b-2,\a-\b\}}))
\end{equation}
with probability at least $1-C\exp(-c(\log d)^{\b-2})$.

According to the Deuring-Heilbronn phenomenon if the Dedekind zeta function $\zeta_K$ of a number field $K$ has a real zero very close to $1$, then it cannot have other zeros nearby $1$. More precisely (see \cite{LMO-Chebotarev}*{Theorem 5.1}) there is a positive, absolute, effectively computable constant $c_0>0$ such that for every number field $K$ if $\zeta_K$ has a real zero $\rho_{K,0}$, then every other zero $\rho$ satisfies:
\[
|1-\rho| \ge \frac{c_0}{ \log(2^d\Delta_K)} \log\big( \frac{c_0}{|1-\rho_{K,0}| \log(2^d\Delta_K)}\big).
\]

So assume, by contradiction, that  $\zeta_K$ has a zero $\rho_{K,0}$
with $|1-\rho_{K,0}|<\exp(-(\log d)^{\a+1})$ for each  $K=\Q(a)$ for each non-zero complex root $a$ of $P$, which is not a root of unity (this is void and hence always holds if $P$ is a product of cyclotomic polynomials or factors of the type $x^m$). Note then that $|1-\rho_{K,0}|<1/(4\log |\Delta_K|)$ (because by Lemma \ref{lm:discriminant} $|\Delta_K|\leq d^{(1+\tau^{-1})2d}$) and hence by \cite{Sta-Dedekind}*{Lemma 3} $\rho_{K,0}$ must be real and is the unique Siegel zero of $\zeta_K$. Thus every other zero $\rho$ satisfies 
\[
|1-\rho|\ge  \frac{c}{d \log d} \log\big( \frac{c \exp( (\log d)^{\a +1})}{d \log d}\big)
\ge \frac{(\log d)^{\a'}}{d}
\]
provided $d$ is sufficiently large.

So we can apply Proposition \ref{pr:Bp-sum} and, using Lemma \ref{lm:weights}, write
\begin{align*}
\sum_{p} B_P(p)\log(p)g_{X,k}(\log p)=&\sum_{O \in \Omega/G}(1-G_{X,k}(\rho_{K_O,0}))+O(\exp(-c(\log d)^{\a'-\b}))\\
\le&|\Omega/G|X\exp(-(\log d)^{\a+1})+C\exp(-c(\log d)^{\a'-\b})\\
\le& d^2\exp(-(\log d)^{\a+1})+C\exp(-c(\log d)^{\a'-\b}).
\end{align*}
But this is incompatible with \eqref{eq:BP-sum-exp-zero}.

\bibliographystyle{abbrv}
\bibliography{bibfile}

\end{document}